\documentclass{article}
\usepackage[utf8]{inputenc}
\usepackage{mathrsfs}
\usepackage{mathtools}
\usepackage[normalem]{ulem}
\usepackage[colorinlistoftodos]{todonotes}
\usepackage{biblatex}
\usepackage{tikz}
\usepackage{graphicx}
\usepackage{tikz-cd}
\usepackage{tikz}
\usepackage{bbm}

\usepackage{dsfont}
\usepackage{url}
\usepackage{hyperref}
\usepackage{amsmath}
\usepackage{amssymb}
\usepackage{mathrsfs}
\usepackage{braket}
\usepackage{mathtools}
\usepackage{booktabs}
\usepackage{caption}
\usepackage{cancel}
\usepackage{epstopdf}
\usepackage{comment}
\usepackage{appendix}

\DeclarePairedDelimiter{\norm}{\lVert}{\rVert}
\newcommand{\numberset}{\mathbb}
\newcommand{\N}{\numberset{N}}
\newcommand{\R}{\numberset{R}}

\newcommand{\Z}{\numberset{Z}}

\newcommand{\sphere}{\mathbb{S}}

\newcommand{\supp}{\text{supp}}
\usepackage{amsthm}
\newtheorem{thm}{Theorem}[section]

\newtheorem{lem}[thm]{Lemma}
\newtheorem{prp}[thm]{Proposition}

\newtheorem{defn}[thm]{Definition}

\newtheorem{rem}[thm]{Remark}
\newtheorem{es}{Example}
\newtheorem{exo}{Exercise}

\makeatletter
\newcommand*{\@old@slash}{}\let\@old@slash\slash
\def\slash{\relax\ifmmode\delimiter"502F30E\mathopen{}\else\@old@slash\fi}
\makeatother

\makeatletter
\def\bign#1{\mathclose{\hbox{$\left#1\vbox to8.5\p@{}\right.\n@space$}}\mathopen{}}
\def\Bign#1{\mathclose{\hbox{$\left#1\vbox to11.5\p@{}\right.\n@space$}}\mathopen{}}
\def\biggn#1{\mathclose{\hbox{$\left#1\vbox to14.5\p@{}\right.\n@space$}}\mathopen{}}
\def\Biggn#1{\mathclose{\hbox{$\left#1\vbox to17.5\p@{}\right.\n@space$}}\mathopen{}}
\makeatother

\usepackage{todonotes}
\usepackage{float}
\theoremstyle{remark}

\renewcommand{\phi}{\varphi}

\title{The completion of the set of Lagrangians \\
and applications to dynamics \\ \Large  Based on lectures by C. Viterbo}
\author{Olga Bernardi\footnote{Dipartimento di Matematica Tullio Levi--Civita, Università di Padova, via Trieste 63, 35121 Padova, Italy. Email: obern@math.unipd.it,}, Francesco Morabito\footnote{D-MATH, ETH, R\"amistrasse 101, 8092, Zurich, Switzerland. Email: francesco.morabito@math. ethz.ch}}

\begin{document}

\maketitle
\section{Introduction}
The goal of these lectures is to introduce the completion of the set of Lagrangian submanifolds of a symplectic manifold with respect to the spectral metric first introduced by V. Humilière in \cite{hum08} and recently revisited in \cite{vit22}.  We establish a number of basic properties of this completion, in particular through the notion of $\gamma$-support, which we develop as a refinement of Humilière's original concept. We then present an application of these notions to conformally symplectic dynamics, generalizing the notion of Birkhoff attractor as defined and studied by G.D. Birkhoff, M. Charpentier, and more recently P. LeCalvez (see \cite{bir32, cha34, lec88}). Finally, we briefly mention several other applications of the Humilière completion and highlight many open questions.

These are notes elaborated from the lectures with the same title given at the CIME School ``Symplectic Dynamics and Topology'' held in Cetraro (CS) Italy from 16 to 20 June 2025. The authors thank Claude Viterbo and Marco Mazzucchelli for clarifying discussions while writing down these notes. It has been a pleasure to take part in this project.

C. Viterbo  warmly thanks the organizers of the CIME conference, Olga Bernardi, Anna Florio, Marco Mazzucchelli and Alfonso Sorrentino, for the perfect organization in a paradisiac location as well as Olga Bernardi and Francesco Morabito who undertook the daunting task of turning a confusing expositions into readable lecture notes. 
The reader should be aware that any shortcomings in these notes are the sole responsibility of the lecturer.

\subsection{Notations}
\begin{itemize}
\item $\mathfrak{L}(M, -d\lambda) = \mathfrak{L}(M)$ is the set of closed, connected, exact Lagrangians in $(M, -d\lambda)$. 
\item $\mathscr{L}(M, -d\lambda) = \mathscr{L}(M) := \Set{(L, f_L): L \in \mathfrak{L}(M) \text{ and } \lambda\vert_{L} = df_L }$ is the set of Lagrangian branes or simply branes. 
\item $\mathfrak{L}_0(T^*N)$ is the set of Lagrangians in $\mathfrak{L}(T^*N)$ which are Hamiltonian isotopic to the zero section of $T^*N$. 
\item ${\mathscr L}_0(T^*N)$ is the set of branes over an element of $\mathfrak L_0(T^*N)$.
\item $\mathrm{DHam}(M)$ is the set of time-1 flows of Hamiltonian vector fields on $(M, -d\lambda)$.
\item $\mathrm{DHam}_c(M)$ the set of time-1 flows of Hamiltonian vector fields for compactly supported Hamiltonians on $(M, -d\lambda)$. 
\item $\widehat{\mathfrak{L}_0}(T^*N)$ is the completion of the metric space $(\mathfrak{L}_0(T^*N),\gamma)$. 
\item $\widehat{\mathscr{L}_0}(T^*N)$ is the completion of the metric space $(\mathscr{L}_0(T^*N),c)$.
\item $\widehat{\mathfrak{L}_c}(T^*N)$ is the set of elements in $\widehat{\mathfrak{L}_0}(T^*N)$ with compact $\gamma$-support.
\item $\widehat{\mathscr{L}_c}(T^*N)$ is the set of elements in $\widehat{\mathscr{L}_0}(T^*N)$ with compact $c$-support.
\item $B_0 = B_0(\psi)$ is the attractor of Definition \ref{Conley}.
\item $B = B(\psi)$ is the Birkhoff attractor in $T^*\mathbb{S}$.
\item $B_{\infty} = B_{\infty}(\psi)$ is the generalized Birkhoff attractor in $T^*N$.
\end{itemize}

\section{Preliminaries} \label{PRE}
\begin{center}
\begin{minipage}{10 cm}
\textit{Contents of Section \ref{PRE}.} We first define Liouville exact symplectic manifolds and conformally symplectic vector fields and diffeomorphisms. We then introduce the main objects of our study: Lagrangian submanifolds, Lagrangian branes and their global description by generating functions. We finally recall the main results on existence and uniqueness of generating functions quadratic at infinity for Lagrangian submanifolds.
\end{minipage}
\end{center}
\subsection{Liouville exact symplectic manifolds}
A smooth manifold $M$ of dimension $2n$ is \textit{symplectic} if it is equipped with a closed and non degenerate $2$-form $\omega \in \Omega^2(M)$. A symplectic manifold $(M,\omega)$ is called \textit{exact symplectic} if $\omega = -d\lambda$ where $\lambda$ is called Liouville form. \\
Throughout the whole paper, we will assume that $(M,-d\lambda)$ is a \textit{Liouville manifold}, that is $(M,-d\lambda)$ falls into one of the following two cases : 

\begin{itemize}
\item[$(a)$] $M$ is a compact manifold with boundary $\partial M$ of contact type, that is:
\begin{itemize}
\item[$(a_1)$] $\lambda \in \Omega^1(M)$ is a \textit{contact 1-form} on $\partial M$, which means that $\lambda \wedge (d\lambda)^{n-1}$ does not vanish on $\partial M$.
\end{itemize}
or, equivalently:
\begin{itemize}
\item[$(a_2)$] the \textit{Liouville vector field} $X_\lambda \in X(M)$, defined by $i_{X_\lambda}\omega = \lambda$, is transverse to the boundary $\partial M$. We notice that $X_{\lambda}$ points outwards $\partial M$.
\end{itemize}
\item[$(b)$] $M$ is an open manifold and there exists a sequence of compact submanifolds of codimension 0 with smooth boundary $M_1 \subset M_2 \subset \ldots$ such that
$$M = \bigcup_{i = 1}^{+\infty} M_i$$ 
and every $\partial M_i$ is of contact type, as defined in case $(a)$. Moreover, we require the Liouville vector field $X_\lambda$ to be complete both in backward and forward time. 
\end{itemize}
\begin{rem}
\textnormal{Every compact Liouville manifold $(M,-d\lambda)$ as in $(a)$ provides an open Liouville manifold as in $(b)$ by taking
$$W = M \cup_{\partial M \times \{0\}} (\partial M \times \mathbb{R}_{\geq 0})$$
equipped with the exact symplectic form coinciding with $-d\lambda$ on $M$ and with $-d(e^t\lambda)$ on $\partial M \times \mathbb{R}_{\geq 0}$. In such a case, the sequence of compact subsets introduced in $(b)$ is given by
$$M_i = M \cup_{\partial M \times \{0\}}  (\partial M \times [0,i])\, ,$$
see Figure \ref{W}. We underline that for the surface with boundary $W$ depicted in Figure \ref{W} (a torus without a disc), since $H^2(W; \R)=0$, the area form on $W$ is exact symplectic.}
\end{rem}

\begin{figure}[ht]
  \centering
    %% Creator: Inkscape 1.2.2 (b0a8486541, 2022-12-01), www.inkscape.org
%% Accompanies image file 'M e W.pdf' (pdf, eps, ps)
\begingroup%
  \makeatletter%
  \providecommand\color[2][]{%
    \errmessage{(Inkscape) Color is used for the text in Inkscape, but the package 'color.sty' is not loaded}%
    \renewcommand\color[2][]{}%
  }%
  \providecommand\transparent[1]{%
    \errmessage{(Inkscape) Transparency is used (non-zero) for the text in Inkscape, but the package 'transparent.sty' is not loaded}%
    \renewcommand\transparent[1]{}%
  }%
  \providecommand\rotatebox[2]{#2}%
  \newcommand*\fsize{\dimexpr\f@size pt\relax}%
  \newcommand*\lineheight[1]{\fontsize{\fsize}{#1\fsize}\selectfont}%
  \ifx\svgwidth\undefined%
    \setlength{\unitlength}{255.19940474bp}%
    \ifx\svgscale\undefined%
      \relax%
    \else%
      \setlength{\unitlength}{\unitlength * \real{\svgscale}}%
    \fi%
  \else%
    \setlength{\unitlength}{\svgwidth}%
  \fi%
  \global\let\svgwidth\undefined%
  \global\let\svgscale\undefined%
  \makeatother%
  \begin{picture}(1,0.57666124)%
    \lineheight{1}%
    \setlength\tabcolsep{0pt}%
    \put(0,0){\includegraphics[width=\unitlength,page=1]{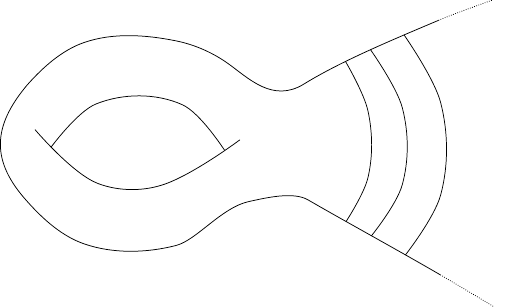}}%
    \put(0.16130692,0.15315305){\color[rgb]{0,0,0}\makebox(0,0)[lt]{\lineheight{1.25}\smash{\begin{tabular}[t]{l}$M$\end{tabular}}}}%
    \put(0.88271254,0.28424109){\color[rgb]{0,0,0}\makebox(0,0)[lt]{\lineheight{1.25}\smash{\begin{tabular}[t]{l}$W$\end{tabular}}}}%
    \put(0.45770446,0.46007609){\color[rgb]{0,0,0}\makebox(0,0)[lt]{\lineheight{1.25}\smash{\begin{tabular}[t]{l}$\partial M \times \{0\}$\end{tabular}}}}%
    \put(0.52214275,0.08785971){\color[rgb]{0,0,0}\makebox(0,0)[lt]{\lineheight{1.25}\smash{\begin{tabular}[t]{l}$\partial M \times \{1\}$\end{tabular}}}}%
    \put(0.60100307,0.52491861){\color[rgb]{0,0,0}\makebox(0,0)[lt]{\lineheight{1.25}\smash{\begin{tabular}[t]{l}$\partial M \times \{2\}$\end{tabular}}}}%
  \end{picture}%
\endgroup%

  \caption{An open Liouville manifold $W = M \cup_{\partial M \times \{0\}} (\partial M \times \mathbb{R}_{\geq 0})$.}
  \label{W}
\end{figure}
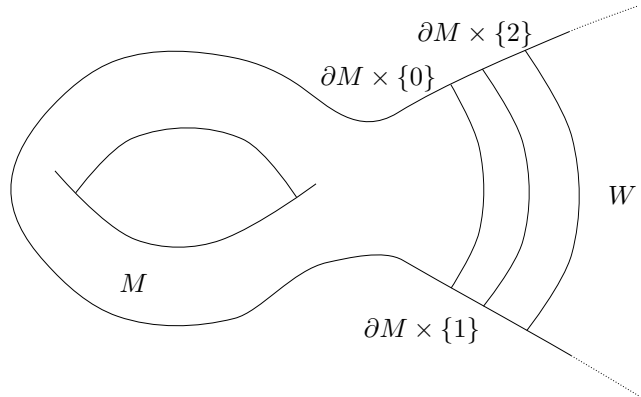

\begin{rem} \label{cot bun}
\textnormal{The main example of case $(b)$ is the cotangent bundle $(T^*N,-d\lambda)$ of any closed smooth manifold $N$ with the standard Liouville form $\lambda = pdq$. We observe that, fixed $r > 0$, we obtain a compact Liouville manifold as in $(a)$ by taking
$$D_rT^*N := \Set{(q, p)\in T^*N: \vert p \vert_{g^*} \leq r}\, ,$$
where $g^*$ is the dual metric induced by a Riemannian metric $g$ on $N$.}
\end{rem}

\subsection{CS vector fields and diffeomorphisms}
\begin{defn}
A vector field $Y \in X(M)$ on a symplectic manifold $(M,\omega)$ is conformally symplectic (CS) if 
\begin{equation} \label{CS}
L_Y\omega = \alpha \omega\, ,
\end{equation}
for some $\alpha \in \mathbb{R}$. When $\alpha = 0$, $Y \in X(M)$ is symplectic.
\end{defn}
\noindent The constant $\alpha \in \mathbb{R}$ is called the \textit{conformal rate}. Clearly, the flow $\varphi_Y^t$ of a conformally symplectic vector field $Y$ satisfies
\begin{equation} \label{CSflow}
(\varphi_Y^t)^* \omega = e^{\alpha t} \omega\, .
\end{equation}
We observe that, a more general notion of conformally symplectic vector field would seem to be
$$L_Y \omega = f \omega\, ,$$
for $f \in C^{\infty}(M;\mathbb{R})$. However, such a definition will not be more general unless $n = 1$ because for $n \geq 2$ the function $f$ must necessarily be a constant. Indeed, since $L_Y\omega$ is closed, we have that
$$df \wedge \omega = 0\, .$$
It is then sufficient to observe that, for $n \ge 2$, the wedge product with $\omega$ is (pointwise) injective and therefore $f$ must be a constant. 
\begin{exo}
\textnormal{What is the set of possible $f$ when $n=1$?}
\end{exo}
\noindent Let $(M,-d\lambda)$ be an exact symplectic manifold. Then the Liouville vector field $X_\lambda \in X(M)$, defined by $i_{X_\lambda}\omega = \lambda$, is conformally symplectic with conformal rate $\alpha = -1$ since by Cartan's formula: 
$$L_{X_\lambda} \omega = d \iota_{X_\lambda} \omega + \underbrace{\iota_{X_\lambda} d \omega}_{=0}  = d\lambda = -\omega\, .$$
More generally, $Y \in X(M)$ is a conformally symplectic vector field on an exact symplectic manifold $(M,-d\lambda)$, with conformal rate $\alpha = -1$, if and only if there exists a closed form $\beta \in \Omega^1(M)$ such that
$$Y = X_\lambda + Z\, ,$$
where $Z \in X(M)$ is defined by $\iota_{Z}\omega = \beta$. Again, since $L_{Z}\omega = L_{Y - X_{\lambda}} \omega = 0$, by Cartan's formula:
$$d\iota_Z \omega = L_Z \omega - \underbrace{\iota_{Z} d \omega}_{=0} = 0\, .$$
Setting $\beta := \iota_Z \omega$, the above statement immediately follows.

In particular, if $\beta$ is exact, that is $\beta =dH$, we denote $X_H := X_\beta$ the corresponding Hamiltonian vector field. 
\begin{defn}
A diffeomorphism $\psi \in \mathrm{Diff}(M)$ on a symplectic manifold $(M,\omega)$ is conformally symplectic (CS) if 
\begin{equation} \label{CSdiffeo}
\psi^* \omega = a \omega\, ,
\end{equation}
for some $a > 0$. When $a = 1$, $\psi$ is symplectic.
\end{defn}
\noindent The constant $a \in \mathbb{R}_{> 0}$ is called the \textit{conformal ratio}. Clearly, by (\ref{CSflow}) the flow $\varphi_Y^t$ of every CS vector field of conformal rate $\alpha$ is a CS diffeomorphism of conformal ratio $e^{\alpha t}$.
\begin{defn}
A diffeomorphism $\psi \in \mathrm{Diff}(M)$ on an exact symplectic manifold $(M,-d\lambda)$ is conformally exact symplectic (CES) with respect to $\lambda$ if 
$$\psi^* \lambda - a \lambda = df\, ,$$
that is $\psi^* \lambda - a \lambda$ is (closed and) exact. When $a = 1$, $\psi $ is exact symplectic.
\end{defn}
\begin{rem} $(i)$ \textnormal{Note that  $\omega$ defines $\lambda$ up to a closed $1$-form, and exactness only depends on the choice of $\lambda$ up to the addition of an exact $1$-form: if $\psi \in \mathrm{Diff}(M)$ is CES with respect to $\lambda$, then it is CES with respect to $\lambda+dg$.} \\
$(ii)$ \textnormal{Given an exact symplectic manifold $(M,-d\lambda)$, let $\psi \in \mathrm{Diff}(M)$ be a CS diffeomorphism (for $a \ne 1$), homotopic to the identity. Then there exists a primitive $\lambda_1$ of $-d\lambda$ such that $\psi \in \mathrm{Diff}(M)$ is CES with respect to $\lambda_1$. We refer to \cite[Appendix B Proposition 9]{arnFéj} for the detailed proof of this fact.}
\end{rem}
\begin{es}
\textnormal{Let us consider a conservative system given by a potential $V$, so the Hamiltonian  is $H(q,p)=\frac{1}{2} \vert p\vert^2+V(q)$ to which we add a friction term proportional to the speed. We then get the equations
    \begin{equation}
       \left\{ \begin{array}{cc}
            \dot q =&p  \\
             \dot p =& -\nabla V(q)-\alpha p 
        \end{array}
\right.
    \end{equation}
    The reader can check that the vector field is conformally symplectically exact, with conformal rate $\alpha > 0$. Indeed, this is the sum of a Hamiltonian vector field and the vector field $X_\alpha(q,p)=(0,-\alpha p)$ and $L_{X_\alpha}(-dp\wedge dq)=-di_{X_\alpha}(dp\wedge dq)=-\alpha dp\wedge dq$.}
\end{es}

\subsection{Lagrangian submanifolds and Lagrangian branes} \label{sezione 23}

\noindent Let $(M,\omega)$ be a symplectic manifold of dimension $2n$ and let $V \subset M$ be an embedded submanifold of dimension $k$. For every $z \in V$ we consider the symplectic orthogonal subspace of $T_zV$ at $z$, given by
$$(T_zV)^{\bot} := \set{ v \in T_zM: \omega(u,v) = 0 \text{ for every } u \in T_zV}\, .$$
Since $\omega$ is nondegenerate, $\textnormal{dim}(T_zV)^{\bot} = 2n -k$. We say that
\begin{itemize}
\item[$(i)$] $V$ is isotropic if $T_zV \subset (T_zV)^{\bot}$ for every $z \in V$;
\item[$(ii)$] $L$ is coisotropic if $(T_zV)^{\bot} \subset T_zV$ for every $z \in V$;
\item[$(iii)$] $V$ is Lagrangian if $(T_zV)^{\bot} = T_zV$ for every $z \in V$.
\end{itemize}
Hence \textit{Lagrangians} are maximal isotropic or minimal coisotropic (embedded) submanifolds. Briefly, $L \subset M$ is Lagrangian if 
$$(i)\, \dim(L) = n \qquad \text{and} \qquad (ii)\, \omega\vert_{TL} = 0 \, .$$
In particular, if $(M,-d\lambda)$ is exact symplectic, $L \subset M$ is \textit{exact Lagrangian} if 
$$(i)\, L \text{ is Lagrangian} \qquad \text{and} \qquad (ii)\, \lambda\vert_{L} \text{ is exact}\, .$$
In this case, there exists a primitive function $f_L \in C^{\infty}(L;\mathbb{R})$ of $\lambda$ on $L$, that is $\lambda\vert_{L} = df_L$.  If $L$ is connected, the primitive $f_L$ is unique up to addition of a constant. \\ \\
\noindent In the sequel, the set of closed, connected, exact (embedded) Lagrangians in $(M, -d\lambda)$ will be denoted by $\mathfrak{L}(M, -d\lambda)$.  Moreover we set
\begin{equation}
    \mathscr{L}(M, -d\lambda) := \Set{(L, f_L): \lambda\vert_{L} = df_L }.
\end{equation}

Of course,  there is a ``forgetful'' map
$$\mathscr{L}(M, -d\lambda) \rightarrow \mathfrak{L}(M, -d\lambda)$$
whose fiber over $L \in \mathfrak{L}(M, -d\lambda)$ is the set of all
$$T_c(L,f_L) := (L,f_L + c)$$
for $c \in \mathbb{R}$. We call an element $(L,f_L) \in \mathscr{L}(M, -d\lambda)$ Lagrangian brane or simply brane. 
\begin{es} \textnormal{Consider the case of a cotangent bundle $(T^*N,-d\lambda)$, as in Remark \ref{cot bun}. The graph of a $1$-form $\alpha \in \Omega^1(N)$:
$$G_\alpha = \set{ (q,\alpha(q)): q \in N} \subset T^*N$$
is Lagrangian if and only if $\alpha$ is closed since $\Lambda_{\mid G_{\alpha}}=\alpha$ (see e.g. \cite[Proposition 3.4.2]{mcDSal16}). Moreover $G_\alpha$ is exact Lagrangian if and only if $\alpha$ is exact. In such a case, given a primitive $f$ of $\alpha$, we write $\Gamma_f := G_\alpha$.}
\end{es}
\begin{es} \textnormal{Let $\mathbb{S}=\mathbb{R}/\mathbb{Z}$. We remind that an essential curve in $T^*\mathbb{S}$ is a topological embedding of $\mathbb{S}$ which is not homotopic to a point. Then an essential curve on $(T^*\mathbb{S}, -pdq)$ is exact Lagrangian if and only if the signed area of the portion of cylinder contained between the curve and the zero section $\mathbb{S} \times \{0\}$ is $0$, see Figure \ref{Lag Cil}. This may be justified remarking that the conclusion is obvious if the Lagrangian in $C^1$--close to the zero section, and then extend it to all Lagrangians following the same path as in Exercise \ref{ex:nlc}.} \newline
\end{es}
\begin{figure}[H]
\centering
 \resizebox{0.7\linewidth}{!}{
%% Creator: Inkscape 1.2.2 (b0a8486541, 2022-12-01), www.inkscape.org
%% Accompanies image file 'cylinderLagrangians.pdf' (pdf, eps, ps)
\begingroup%
  \makeatletter%
  \providecommand\color[2][]{%
    \errmessage{(Inkscape) Color is used for the text in Inkscape, but the package 'color.sty' is not loaded}%
    \renewcommand\color[2][]{}%
  }%
  \providecommand\transparent[1]{%
    \errmessage{(Inkscape) Transparency is used (non-zero) for the text in Inkscape, but the package 'transparent.sty' is not loaded}%
    \renewcommand\transparent[1]{}%
  }%
  \providecommand\rotatebox[2]{#2}%
  \newcommand*\fsize{\dimexpr\f@size pt\relax}%
  \newcommand*\lineheight[1]{\fontsize{\fsize}{#1\fsize}\selectfont}%
  \ifx\svgwidth\undefined%
    \setlength{\unitlength}{212.5984252bp}%
    \ifx\svgscale\undefined%
      \relax%
    \else%
      \setlength{\unitlength}{\unitlength * \real{\svgscale}}%
    \fi%
  \else%
    \setlength{\unitlength}{\svgwidth}%
  \fi%
  \global\let\svgwidth\undefined%
  \global\let\svgscale\undefined%
  \makeatother%
  \begin{picture}(1,0.73333333)%
    \lineheight{1}%
    \setlength\tabcolsep{0pt}%
    \put(0,0){\includegraphics[width=\unitlength,page=1]{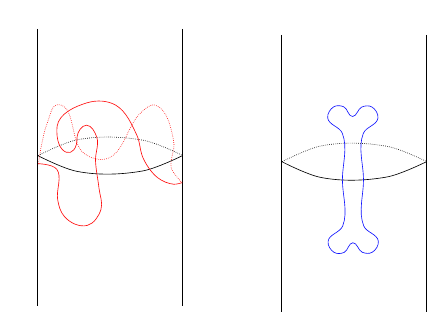}}%
    \put(0.25067862,0.29648718){\color[rgb]{0,0,0}\makebox(0,0)[lt]{\lineheight{1.25}\smash{\begin{tabular}[t]{l}$\mathcal{O}_{\sphere}$\end{tabular}}}}%
    \put(0.66328486,0.29320711){\color[rgb]{0,0,0}\makebox(0,0)[lt]{\lineheight{1.25}\smash{\begin{tabular}[t]{l}$\mathcal{O}_{\sphere}$\end{tabular}}}}%
    \put(0,0){\includegraphics[width=\unitlength,page=2]{cylinderLagrangians.pdf}}%
  \end{picture}%
\endgroup%
}
\caption{On the left, an exact Lagrangian in $(T^*\mathbb{S}, -pdq)$: blue and red areas sum to zero. On the right, a non-essential curve surrounding a non-zero area.}
\label{Lag Cil}
\end{figure}
\noindent Given a Hamiltonian $H \in C^\infty(\mathbb{S} \times M; \mathbb{R})$, its Hamiltonian vector field is defined by 
 $$i_{X_{H_t}} \omega = dH_t\, ,$$
where $H_t(\cdot) := H(t,\cdot)$. \\
\\
\noindent In the sequel, we indicate by $\mathrm{DHam}(M)$ the set of time-1 flows of Hamiltonian vector fields on $(M, -d\lambda)$ and by $\mathrm{DHam}_c(M)$ the set of time-1 flows of Hamiltonian vector fields generated by compactly supported Hamiltonians on $(M, -d\lambda)$. We call the elements in $\mathrm{DHam}(M)$ {\textit{Hamiltonian diffeomorphisms} and the elements in $\mathrm{DHam}_c(M)$  {\textit{compactly supported Hamiltonian diffeomorphisms}. 
It is easy to show that the exactness of a Lagrangian is preserved under Hamiltonian diffeomorphisms, see \cite[Corollary 9.3.6]{mcDSal16}. This means that $\mathrm{DHam}(M)$ sends $\mathfrak{L}(M, -d\lambda)$ into itself. This action lifts to $\mathscr{L}(M, -d\lambda)$ as explained here below. \\
Let $(\varphi^t_H)_{t \in \mathbb{R}}$ be the Hamiltonian flow for $H$. Let $\varphi_H \in \mathrm{DHam}(M)$ be the corresponding time-1 flow. Then
$$(L,f_L) \mapsto(\varphi_H(L), f_{\varphi_H(L)})$$
where
$$f_{\varphi_H(L)}(z) = f_L(\varphi_H^{-1}(z)) + \int^1_0[\gamma^*\lambda + H(t,\gamma(t))]dt\, ,$$
where $\gamma(t) := \varphi^t_H(\varphi_H^{-1}(z))$.
\begin{rem} \label{nota supp com} {\textnormal{In the argument above, as well as in other cases later on, $\varphi_H$ can be assumed to be a compactly supported Hamiltonian diffeomorphism. In fact, since $L$ is compact, given $\varphi_H \in \mathrm{DHam}(M)$, there exists $\psi_K \in \mathrm{DHam}_c(M)$ such that $\varphi_H(L) = \psi_K(L)$, with Hamiltonian $K$ itself compactly supported. It is worth noting that, when $M$ is an open manifold (as in the case where $M = T^*N$) the Hamiltonian $K$ is then uniquely defined by $\psi_K \in \mathrm{DHam}_c(M)$. On the other hand, when $M$ is compact, the Hamiltonian $K$ is defined up to an additive constant, which could be time dependent.}}
\end{rem}}

\subsection{Generating functions}
The aim of the present section is to motivate and then introduce generating functions for exact Lagrangians in cotangent bundles. \\ \\
\noindent We recall that every exact Lagrangian $L \subset T^*N$ which is $C^1$-close to the zero section $\mathcal{O}_N$ of $T^*N$ is the graph of the differential of a function $f$ on $N$, that is
$$L = \Gamma_f = \set{ (q,df(q)): q \in N}\, .$$
In such a case, there is a $1:1$ correspondence between points in $\Gamma_f \cap \mathcal{O}_N$ and critical points of $f$:
$$\textnormal{Crit}(f) := \set{q \in N: df(q) = 0 }\, .$$
\noindent In order to discuss an extension of the above correspondence:
$$\Gamma_f \cap \mathcal{O}_N \leftrightarrow \textnormal{Crit}(f)$$ 
to Hamiltonian diffeomorphisms, we need the following general result. In what follows, $(M,-d \lambda)$ is a Liouville manifold and $\Delta = \set{(z,z): z \in M}$ is the diagonal of $M\times M$. Moreover, we denote by $\lambda_\Delta$ the canonical $1$-form on $T^*\Delta$.

{\begin{prp}
Let $\varphi\in\mathrm{DHam}_c(M)$ be $C^1$--close to the identity. Then
$$\Gamma_\varphi = \set{(z,\varphi(z)): z \in M}$$
is the graph of an exact $1$-form on $\Delta$.
\end{prp}

\begin{proof} Let $\textnormal{Symp}_0(M)$ be the identity component of the group of symplectomorphisms $\textnormal{Symp}(M)$ and $\widetilde{\textnormal{Symp}_0}(M)$ its universal covering. Let $\varphi_t : M \to M$ be a compactly supported symplectic isotopy. We first recall --see \cite[Theorem 10.2.5]{mcDSal16}-- that $\varphi_1$ is the time-1 flow of a Hamiltonian vector field if and only if $\{\varphi_1\}$ is in the kernel of the Flux homomorphism: $\widetilde{\textnormal{Symp}_0}(M) \to H^1_c(M;\mathbb R)$. 

When the symplectic manifold $(M,-d\lambda)$ is exact, the Flux is defined by $\textnormal{Flux}(\{\varphi_t\}) =  [\lambda-\varphi^*_1\lambda]$ (see \cite[Lemma 10.2.7]{mcDSal16}. Moreover, Weinstein's theorem \cite[Theorem 3.4.13]{mcDSal16} asserts that a neighborhood of the diagonal in $(M \times M, -d(\lambda\ominus \lambda))$ is symplectomorphic to a neighborhood of the zero section in $(T^*\Delta, -d\lambda_\Delta)$. 

Moreover  by Poincaré's lemma, we see that the pull-back by the symplectomorphism  of  $\lambda_\Delta$ differs from $\lambda\ominus \lambda$ by an exact form.
Since $\varphi$ is $C^1$-close to the identity, $\Gamma_\varphi= (\textnormal{Id}\times\varphi)(\Delta)$ is identified by the symplectomorphism to the graph of a $1$-form $\theta$ on $\Delta$, and $\varphi$ being symplectic, $\theta$ is closed. Finally, since $\varphi \in \textnormal{DHam}_c(M)$, the Flux of $\phi$ is then $$0 = [\lambda - \varphi^*\lambda] = [ (\lambda \ominus \lambda)|_{\Gamma_\varphi}] = [{\lambda_{\Delta}}_{\mid \Gamma_{\varphi}}]=[\theta]$$ which means that $\theta$ is exact. 
\end{proof}

\begin{rem}
\textnormal{The same result can be formulated in a more general --not necessarily Liouville-- context. However, it is worth noting that, out of the Liouville case, one has to be careful that on top of the crucial $C^1$-closeness to the identity condition of $\varphi=\varphi_1$, we need to suppose that the whole isotopy $\{\varphi_t\}$ remains in such $C^1$-small neighborhood of the identity. In fact, one can find a Lagrangian $L$ and a $C^0$-small Hamiltonian isotopy $\{\varphi_t\}$ such that $\varphi_t(L)$ is in a Weinstein neighbourhood of $L$ but is not exact as a submanifold in $T^*L$. In particular, this is possible for $L=\mathbb{T}^n$ and in any symplectic manifold of dimension $\geq 6$. For this cases and the connection with the $C^0$ and $C^1$ Flux conjecture we refer to \cite{ACLS}.}
\end{rem}
\noindent Under the hypotheses of the above proposition, let $f_{\varphi}:\Delta\rightarrow \R$ such that 
$\Gamma_\varphi = \Gamma_{f_\varphi}$. It then follows that $\Gamma_\varphi \cap \Delta \leftrightarrow \text{Crit}(f_{\varphi})$ or, equivalently:
$$\text{Fix}(\varphi) \leftrightarrow \text{Crit}(f_{\varphi}) $$
\noindent It is clear that --even if we restrict to exact Lagrangians in cotangent bundles $T^*N$-- it is far from being true that they may be represented as graphs of differentials of smooth functions. \\
\indent In order to go beyond the perturbative setting explained above, the idea is to ``trade dimension for complexity'', by introducing auxiliary parameters. This strategy justifies the next fundamental definition.
\begin{defn}
$S \in C^\infty(N\times \R^k;\R)$ is a generating function (GF) if
\begin{itemize}
\item[$(i)$] $N\times \R ^k \ni (q;\xi) \mapsto \partial_\xi S(q;\xi)\in \R^k$ has $0$ as regular value;
\item[$(ii)$] $i_S: (q;\xi) \mapsto (q,\partial_q S(q,\xi))$, defined on 
$$\Sigma_S := \set{ (q;\xi): \partial_\xi S(q;\xi) = 0}\, ,$$ 
is an embedding.
\end{itemize}
\end{defn}
\noindent From now on, we indicate $\mathscr{L}(T^*N,-d(pdq))$ [resp. $\mathfrak{L}(T^*N,-d(pdq))$] simply by $\mathscr{L}(T^*N)$ [resp. $\mathfrak{L}(T^*N)$]. Moreover, for the sake of readability, we often omit the choice of the primitive for
a given brane.
\begin{prp}\label{prp:GFQILagEmb}
$i_S(\Sigma_S) \in \mathscr{L}(T^*N)$.
\end{prp}
\begin{proof}
It is enough to provide a primitive of the restriction to $i_S(\Sigma_S)$ of the Liouville form $pdq$:
    \begin{equation*}
        pdq\vert_{i_S(\Sigma_S)} = \partial_qS dq = \partial_qS dq + \partial_\xi Sd\xi = dS\vert_{\Sigma_S}.
    \end{equation*}
\end{proof}
\noindent In the sequel, we use the notation $L_S$ for $i_S(\Sigma_S)$ and we say that $L_S$ is generated by $S$. From the previous proposition, we obtain that the correspondence 
$$L_S \cap \mathcal{O}_N \leftrightarrow \text{Crit}(S)$$
still holds. Clearly, if $k = 0$, $L_S = \Gamma_{S}$.
\begin{rem}
{\textnormal{We remind that, general (that is, not necessarily exact) Lagrangians in a cotangent bundle, can always be described at least locally by a generating function with auxiliary parameters. This is the content of a theorem proved by V. Maslov in \cite{Masl} in 1965 and refined by L. H\"ormander in \cite{Hor} in 1971, we refer to \cite[Theorem 2.1]{Car} for a detailed proof.}}     
\end{rem}
\indent {In order to apply the Calculus of Variations to generating functions, one needs a condition assuring the existence of critical points. The next class of generating functions has traditionally been used\footnote{More recently, \cite{Abouzaid-Courte-Guillermou-Kragh} used a different kind of generating functions with milder conditions at infinity. }: 
\begin{defn} \label{QI}
A generating function $S \in C^\infty(N\times \R^k;\R)$ is quadratic at infinity (GFQI) if for $|\xi| \ge R$
$$S(q;\xi) = \xi^T Q \xi\, ,$$
where $\xi^T Q \xi$ is a non degenerate quadratic form, i.e. $Q$ is a non-degenerate symmetric matrix. 
\end{defn}
{\indent Sometimes \textit{quadratic at infinity} is replaced by \textit{asymptotically quadratic}, we refer to \cite[Definition 2.3]{Theret}. 
\begin{defn} \label{AQI}
A generating function $S \in C^\infty(N\times \R^k;\R)$ is asymptotically quadratic if for every $q \in N$
$$\norm{S(q, \cdot) - Q(q,\cdot) }_{C^1} < +\infty\, , $$
where, fixed $q \in N$, $Q(q,\xi) = \xi^T Q(q) \xi$ is a nondegenerate quadratic.
\end{defn}
\noindent The advantage of the above definition is that the class of exact Lagrangians generated by asymptotically quadratic functions is stable under product, that is
$$L_{S_1} \times L_{S_2} := \set{ (q_1,q_2,p_1,p_2): (q_1,p_1) \in L_{S_1} \text{ and } (q_2,p_2) \in L_{S_2}} \subset T^*(N \times N)$$
has asymptotically quadratic generating function given by 
$$S(q_1,q_2; \xi_1,\xi_2) := S_1(q_1;\xi_1) + S(q_2;\xi_2)\, .$$
The next lemma --see \cite[Proposition 2.12]{the99}-- assures that Definitions \ref{QI} and \ref{AQI} are in fact equivalent.
\begin{lem} \label{serve dopo 1}
Let $S \in C^\infty(N\times \R^k;\R)$ be an asymptotically quadratic generating function. Then there is a GFQI $S' \in C^\infty(N\times \R^k;\R)$ such that $L_S = L_{S'}$.
\end{lem}}}
\indent There are (see e.g. \cite{Hor} and \cite{Wei}) two main operations on GFQIs which preserve the corresponding Lagrangian brane in $\mathscr{L}(T^*N,-d\lambda)$. 
\begin{itemize}
\item[$(i)$] \textit{Stabilization.} Let $S \in C^{\infty}(N\times \mathbb{R}^k,\mathbb{R})$ be a GFQI. Then
$$S'(q;\xi,\eta) := S(q;\xi) + \eta^T R \eta\, ,$$
where $\eta \in \mathbb{R}^l$ and $\eta^T R \eta$ is a nondegenerate quadratic form, is such that $L_S = L_{S'}$.
\item[$(ii)$] \textit{Fibered diffeomorphism.}  Let $S \in C^{\infty}(N\times \mathbb{R}^k,\mathbb{R})$ be a GFQI and $N \times \mathbb{R}^k \ni (q,\xi) \mapsto (q,\phi(q,\xi)) \in N \times \mathbb{R}^k$ be a map such that, for every $q \in N$,
$$\mathbb{R}^k \ni \xi \mapsto \phi(q,\xi) \in \mathbb{R}^k$$
is a diffeomorphism. Then 
$$S'(q;\xi) := S(q;\phi(q,\xi))$$
is such that $L_S = L_{S'}$.
\end{itemize}
\indent Some of the central problems in the global theory of Lagrangian submanifolds are $(a)$ the existence of a GFQI for a Lagrangian submanifold $L \subset T^*N$, $(b)$ the uniqueness of it (up to the operations $(i)$ -- $(ii)$ described above). \\
The next theorem, see \cite{Sik}, partially answers the first question.
\begin{thm}[Laudenbach, Sikorav] \label{thm:lauSik}
If $L \in \mathscr{L}(T^*N)$ admits a GFQI and $\varphi_H \in \mathrm{DHam}(T^*N)$, then $\varphi_H(L)$ admits a GFQI.
\end{thm}
\noindent It is worth noting that, if $L \subset T^*N$ has GFQI, $S(q,\xi)$, with $(q,\xi) \in N \times \mathbb{R}^n$ then a ``formal GFQI'' (formal, since it has infinite parameters) for $\varphi_H(L)$ is given by using the Action Functional, as follows. 
Let $E = H^1([0,1], T^*N)$ be the usual Sobolev space and let $\gamma(s) = (q(t),p(t))$ be a function in $E$. Set
\begin{equation*}
\mathscr{A}_{H,S}(\gamma, \xi):= S(q(0), \xi) + \int_0^1 \left[p(t)\dot{q}(t) - H(t, \gamma(t))\right]dt\, 
\end{equation*}
Then $\mathscr{A}_{H,S}$ can be considered as a generating function with ``auxiliary parameters'' $(\{\gamma(t): t \in [0,1[\},p(1),\xi)$. We then search for the variation of $\mathscr{A}_{H,S}$ on $\set{\gamma \in E:q(1) = q} \times \mathbb{R}^k$: 
\begin{equation}
\begin{split}
& d \mathscr{A}_{H,S}(q; \gamma, \xi)(\delta \gamma,\delta\xi) = \\
& = \frac{\partial S}{\partial \xi} \delta \xi + \frac{\partial S}{\partial q(0)} \delta q(0) + \int^1_0 \left( \dot{q} - \frac{\partial H}{\partial p}\right)\delta p(t)dt + \int^1_0\left( \frac{\partial}{\partial q}(p(t)\dot{q}(t)) - \frac{\partial H}{\partial q} \right) \delta q(t) dt = \\
& = \frac{\partial S}{\partial \xi} \delta \xi + \frac{\partial S}{\partial q(0)} \delta q(0) + \int^1_0 \left( \dot{q} - \frac{\partial H}{\partial p}\right)\delta p(t)dt - \int^1_0 \left(\dot{p} + \frac{\partial H}{\partial q} \right) \delta q(t) dt + \\
& + p(1) \delta q(1) - p(0) \delta q(0)\, .
\end{split}  
\end{equation}
We first notice that $\delta q(1) = 0$ since $q(1) = q$. Therefore, the Lagrangian generated by $\mathscr{A}_H$ is defined by equations:
$$\frac{\partial S}{\partial \xi}(q(0),\xi) = 0, \qquad p(0) = \frac{\partial S}{\partial q(0)} \Rightarrow (q(0),p(0)) \in L$$
and
$$(q(t),p(t)) = \varphi_H^t(q(0),p(0)), \qquad \forall t \in [0,1]\, ,$$
with $\varphi_H^1 = \varphi_H$. If it wasn't for the fact that $\mathscr{A}_{H,S}$ is defined on an infinite diemnsional space, it would qualify as a GFQI for $\varphi_H(L)$.  A reduction to finite parameters (\`a la Amann-Conley-Zehnder \cite[Section 3]{ACZ}, \cite[(2.1.4)]{Cha1} or with the ``broken geodesics'' method \cite{Cha2}) can be performed in order to obtain a ``bona fide' GFQI. \\ 
\indent We observe that Theorem \ref{thm:lauSik} as well as the uniqueness theorem from \cite{vit92} have been reformulated as a fibration theorem by D. Th\'eret ( \cite[top of Page 256, Theorem 4.2]{the99}, as explained here below. \\
For every $k \ge 0$, we call $\mathscr{G}_k(N)$ the set of GFQI on $N \times \mathbb{R}^k$; it is
an open subset of $C^{\infty}(M\times \mathbb{R}^k, \mathbb{R})$ endowed with the Whitney strong topology. Moreover, let
$$\mathscr{G}(N) := \set{ S \in \mathscr{G}_k(N) \text{ for some } k \ge 0} \slash \sim$$
where $S \sim S_1$ if $S$ coincides with $S_1$ up to the two operations on generating functions $(i)$ -- $(ii)$ described above. Of course, there is a well defined map
\begin{equation*}\label{eq:serFib}
\begin{split}
\pi: \mathscr{G}(N) & \rightarrow \mathscr{L}(T^*N) \\
S &\mapsto L_S\, .
\end{split}
\end{equation*}
In the following, $\Delta_n$ denotes the standard $n$-simplex in $\mathbb{R}^{n+1}$. We say that a map $\Delta_n \rightarrow \mathscr{G}(N)$ is smooth if it factors through a smooth function $\Delta_n \rightarrow \mathscr{G}_k(N)$ (for some $k\ge 0$) making the following diagram commutative:
    \begin{equation*}
        \begin{tikzcd}
\Delta_n \arrow[r]\arrow[dr]&\mathscr{G}_k(N)\arrow[d]\\
&\mathscr{G}(N)
\end{tikzcd}.
    \end{equation*}
    
\begin{thm}[Théret]\label{thm:theSerFib}
The map $\pi: \mathscr{G}(N) \rightarrow \mathscr{L}(T^*N)$ is a smooth Serre fibration.
\end{thm}
\noindent The result contained in the above theorem can be rephrased as follows. If the smooth map $f: \Delta_n \to \mathscr{L}(T^*N)$ has a smooth lift $F: \Delta_n \to \mathscr{G}(N)$ and if $(f_t:\Delta_n \to \mathscr{L}(T^*N))_{t \in [0,1]}$ is a smooth homotopy of $f = f_0$, then there is a smooth homotopy $(F_t:\Delta_n \to \mathscr{G}(N))_{t \in [0,1]}$ such that $F_0 = F$ (up to equivalence $\sim$), and $\pi \circ F_t = f_t$ for every $t \in [0,1]$. In particular, the following diagram commutes:
\begin{equation*}
    \begin{tikzcd}
& \Delta_n \times [0,1] \arrow[r] \arrow[dr, dashed] & \mathscr{L}(T^*N) \\
& \Delta_{n} \times \{0\} \arrow[u, hook] \arrow[r] & \mathscr{G}(N) \arrow[u,"\pi"]
\end{tikzcd}.
\end{equation*}
It is not yet known whether $\pi$ is surjective: this would follow from the ``Nearby Lagrangian Conjecture'' attributed to V.I. Arnold (and first stated explicitly in \cite{Lalonde-Sikorav}). \\

\noindent \textbf{Conjecture.} Let $L \in \mathscr{L}(T^*N)$, where $L$ and $N$
are closed manifolds. Then there exists $\varphi_H \in \mathrm{DHam}(T^*N)$ such that $L = \varphi_H(\mathcal{O}_N)$. \\

\noindent However, by the previous theorem, $\pi$ is surjective on any path-connected component of $\mathscr{L}(T^*N) $ which has non-empty intersection with the image of $\pi$.
\begin{exo} \textnormal{Let $(M,-d\lambda)$ be an exact symplectic manifold. Prove that, for any fixed exact Lagrangian $L \in \mathscr{L}(M,-d\lambda)$, the map 
\begin{equation*}
\begin{split}
\mathrm{DHam}_c(M) & \rightarrow \mathscr{L}(M,-d\lambda) \\
\varphi & \mapsto \varphi_H(L)\, .
\end{split}
\end{equation*}
is a smooth Serre fibration.}
\end{exo}
\noindent The fact that the map $\pi: \mathscr{G}(N) \rightarrow \mathscr{L}(T^*N)$ is a smooth Serre fibration, combined to Theorem \ref{thm:lauSik}, provides the next uniqueness (up to equivalence $\sim$) result, see \cite[Theorem 5.1]{the99}.
\begin{thm}[Théret] \label{theret2}
If $L \in \mathscr{L}(T^*N)$ admits a unique GFQI and $\varphi_H \in \mathrm{DHam}(T^*N)$, then $\varphi_H(L)$ admits a unique GFQI.  
\end{thm}
\noindent Finally, let $N$ be a closed smooth manifold. Since --by \cite[Lemma 6.2]{the99}-- the GFQIs for $\mathcal{O}_N$ are all equivalent, Theorem \ref{theret2} applies and gives the next existence and uniqueness result by C. Viterbo. We refer to \cite[Proposition 1.5]{vit92} for a direct proof of this theorem.
\begin{thm}[Viterbo] \label{Vit 92}
If $\varphi_H \in \mathrm{DHam}(T^*N)$, then $L = \varphi_H(\mathcal{O}_N)$ admits a unique GFQI in $\mathscr{G}(N)$.  
\end{thm}

\section{Spectral invariants}\label{SPEC}
\begin{center}
\begin{minipage}{10 cm}
\textit{Contents of Section \ref{SPEC}}.
We introduce the geometry on the space of exact Lagrangian submanifolds (isotopic to the zero section of the cotangent bundle of a closed smooth manifold) and corresponding Lagrangian branes by using spectral invariants. These are numbers that --thanks to Lusternik-Schirelman's theory-- can be extracted from generating functions quadratic at infinity associated to a Lagrangian submanifold. The rest of the section is devoted to prove that these spectral invariants give rise to distances on the spaces of exact Lagrangian submanifolds and Lagrangian branes --denoted respectively by $(\mathfrak{L}_0(T^*N),\gamma)$ and $(\mathscr{L}_0(T^*N),c)$-- where the group of compactly supported Hamiltonian diffeomorphisms acts by isometries.
\end{minipage}
\end{center}

\subsection{Calculus of critical values} \label{PS critical}
Let $f \in C^{\infty}(M,\mathbb{R})$ be a smooth function on a Riemannian manifold $(M, g)$. If $M$ is not compact, we suppose that $f$ satisfies the following weaker condition.
\begin{defn} [Palais-Smale (PS) condition] Every sequence $(x_n)_{n \ge 1}$ such that
$$\vert{df(x_n)}\vert \to 0 \qquad {and} \qquad f(x_n) \text{ is bounded\, ,}$$
admits a converging subsequence.  
\end{defn}
\noindent We notice that the limit of $(x_n)_{n \ge 1}$ must obviously be
a critical point. Clearly, when $M$ is compact, any function is PS. If the PS condition is verified, even if the $M$ is non compact, for every choice of $a < b$ the set of critical points in $f^{-1}([a,b])$ is compact.
\begin{es}
\textnormal{The function $f(x) = \arctan x$ is not PS: $x_n = n$ is a PS sequence, but it does not admit converging subsequences, see Figure \ref{Arc Tan}.}
\end{es}

\begin{figure}
\centering
\includegraphics[height=4cm, width=5cm,
keepaspectratio]{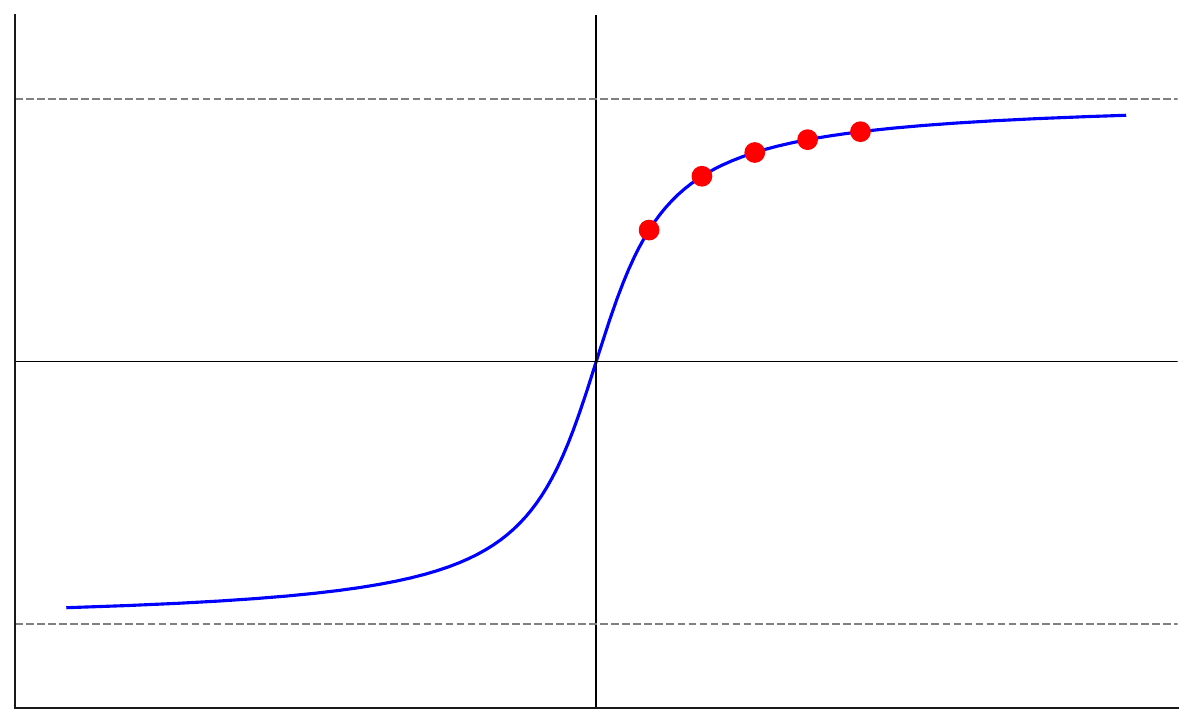}
\caption{$f(x) = \arctan x$ is not PS for the Euclidean metric on $\R^2$.}
\label{Arc Tan}
\end{figure}

A fundamental property of GFQIs is that they are PS for the product metric on $N \times \mathbb{R}^k$.

\begin{lem}
Let $S\in C^{\infty}(N \times \mathbb{R}^k,\mathbb{R})$, $(q,\xi) \mapsto S(q;\xi)$ be a GFQI. Then $S$ is PS.
\end{lem}
\begin{proof}
Let $(q_n, \xi_n)_{n \ge 1}$ be a sequence satisfying
$$\vert{dS(q_n,\xi_n)}\vert \to 0 \qquad {and} \qquad S(q_n,\xi_n) \text{ bounded}\, .$$
If the sequence $(\xi_n)_{n \ge 1}$ is contained in a compact set, then $(q_n, \xi_n)_{n\ge 1}$ admits 
a converging subsequence (because $N$ is compact). Let verify that this is the unique possible case. Remember that $S(q_n,\xi_n) = \xi_n^T Q \xi_n$ for $|\xi_n| \ge R$. It there were only finite terms of the sequence $(\xi_n)_{n \ge 1}$ in every compact set, it would follow that $\lim_{n \to +\infty} |\xi_n| = +\infty$. But this is not possible because, since $\partial_{\xi}S(q,\xi_n) = 2 Q \xi_n$ (for $|\xi_n| \ge R$) and $Q$ is non degenerate, $\vert{dS(q_n,\xi_n)}\vert$ would tend to $+\infty$ and not to $0$.
\end{proof}
\noindent From now on, given $t \in \mathbb{R}$, 
$$f^t := \set{ x \in M: f(x) \le t}\, ,$$ 
and we omit the choice of the Riemannian metric of the manifold except in case of ambiguity. 
\begin{defn} [min-max] For every $0 \ne \alpha \in H^*(f^b,f^a)$, we set
\begin{equation} \label{min-max}
c(\alpha,f) := \inf \set{t \in [a,b]: i^*_{t} \alpha \ne 0}\, ,
\end{equation}
where $i_t^*: H^*(f^b,f^a) \to H^*(f^{t},f^a)$.     
\end{defn}

\noindent The above definition is useful in order to determine critical values of $f$. Indeed --following ideas going back to Poincar\'e and Birkhoff-- we have:

\begin{prp}
If $f \in C^{\infty}(M,\mathbb{R})$ is PS then $c(\alpha,f)$ defined by (\ref{min-max}) is a critical value of $f$, that is $f^{-1}(c(\alpha,f))$ contains a critical point.
\end{prp}
\begin{proof} Assume by contradiction that $c := c(\alpha,f)$ is not critical. Since the set of critical points of $f$ in $f^{-1}([a,b])$ is closed there exists $\varepsilon > 0$ such that $[c-\varepsilon,c+\varepsilon]$ does not contain critical values of $f$. In such a case, as in \cite[Theorem 3.2]{mil63} we push $f^{c + \varepsilon}$ down to $f^{c-\varepsilon}$ along the  trajectories of the ``normalized gradient'' of $f$.
Define the vector field on $f^{c+\varepsilon}\setminus f^{c-\varepsilon}$:
    \begin{equation} \label{normalized gradient}
        X(x):=-\frac{\nabla f(x)}{\vert\nabla f(x)\vert^2}\,,
    \end{equation}
ad extend it smoothly by zero on the rest of $M$. See Figure \ref{milnorField}.
Thanks to the PS condition $\vert X(x)\vert= \frac{1}{\vert \nabla f(x)\vert} $ is bounded, hence its flow $\varphi^t_X$ is complete. Moreover, we have:
    \begin{equation}
        \frac{d}{dt}f(\varphi^t_X(x))=df(X(\varphi^t_X(x))=-1\, ,
    \end{equation}
so that $f(\varphi^t_X(x)) = f(x) -t$. This means that
$$\varphi^{2\varepsilon}_X(f^{c+\varepsilon}) = f^{c -\varepsilon} \quad \Rightarrow \quad H^*(f^{c+\varepsilon}, f^{a}) \cong H^*(f^{c-\varepsilon}, f^a)\, .$$
It then follows that $i^*_{c-\varepsilon} \alpha = 0$ implies $i^*_{c+\varepsilon} \alpha = 0$ so that $c(\alpha,f) \ge c + \varepsilon$. This contradicts the definition of $c$. 
\end{proof}
\begin{figure}[ht]
    \centering
    %% Creator: Inkscape 1.2.2 (b0a8486541, 2022-12-01), www.inkscape.org
%% Accompanies image file 'milnorField.pdf' (pdf, eps, ps)
\begingroup%
  \makeatletter%
  \providecommand\color[2][]{%
    \errmessage{(Inkscape) Color is used for the text in Inkscape, but the package 'color.sty' is not loaded}%
    \renewcommand\color[2][]{}%
  }%
  \providecommand\transparent[1]{%
    \errmessage{(Inkscape) Transparency is used (non-zero) for the text in Inkscape, but the package 'transparent.sty' is not loaded}%
    \renewcommand\transparent[1]{}%
  }%
  \providecommand\rotatebox[2]{#2}%
  \newcommand*\fsize{\dimexpr\f@size pt\relax}%
  \newcommand*\lineheight[1]{\fontsize{\fsize}{#1\fsize}\selectfont}%
  \ifx\svgwidth\undefined%
    \setlength{\unitlength}{110.35842295bp}%
    \ifx\svgscale\undefined%
      \relax%
    \else%
      \setlength{\unitlength}{\unitlength * \real{\svgscale}}%
    \fi%
  \else%
    \setlength{\unitlength}{\svgwidth}%
  \fi%
  \global\let\svgwidth\undefined%
  \global\let\svgscale\undefined%
  \makeatother%
  \begin{picture}(1,0.87033792)%
    \lineheight{1}%
    \setlength\tabcolsep{0pt}%
    \put(0.66371378,0.4067327){\color[rgb]{0,0,0}\makebox(0,0)[lt]{\lineheight{1.25}\smash{\begin{tabular}[t]{l}$f=c$\end{tabular}}}}%
    \put(0.69094333,0.56410618){\color[rgb]{0,0,0}\makebox(0,0)[lt]{\lineheight{1.25}\smash{\begin{tabular}[t]{l}$f=c+\varepsilon$\end{tabular}}}}%
    \put(0.71347265,0.24396058){\color[rgb]{0,0,0}\makebox(0,0)[lt]{\lineheight{1.25}\smash{\begin{tabular}[t]{l}$f=c-\varepsilon$\end{tabular}}}}%
    \put(0,0){\includegraphics[width=\unitlength,page=1]{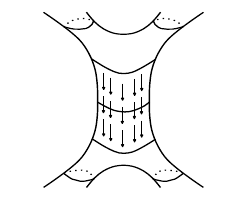}}%
    \put(0.29201078,0.46201156){\color[rgb]{0,0,0}\makebox(0,0)[lt]{\lineheight{1.25}\smash{\begin{tabular}[t]{l}$X$\end{tabular}}}}%
    \put(0,0){\includegraphics[width=\unitlength,page=2]{milnorField.pdf}}%
    \put(-0.00510778,0.71107596){\color[rgb]{0,0,0}\makebox(0,0)[lt]{\lineheight{1.25}\smash{\begin{tabular}[t]{l}$f$\end{tabular}}}}%
  \end{picture}%
\endgroup%

    \caption{The vector field $X$, transverse to the level hypersurfaces of $f$.}
    \label{milnorField}
\end{figure}

\noindent We shall now refine the above proposition to the Lusternik-Schnirelmann theorem (we refer e.g. to \cite[Proposition 2.2]{vit92}). To this aim, we need to specify which cohomology theory we use, namely the Alexander-Spanier cohomology (see \cite[p. 307]{Spanier}), whose main feature is that  the cohomology of a closed set is the limit of the cohomologies of its open neighbourhoods:
\begin{equation*}
    H^*(K):=\lim_{\substack{\longrightarrow \\ U\supset K \\ U \ \mathrm{open}}} H^*(U)\, .
\end{equation*}
We indicate by $\alpha\cup\beta$ the cup product of the cohomology classes $\alpha$ and $\beta$.
\begin{thm} [Lusternik-Schnirelmann] \label{thm:lusSch}
Let $f \in C^{\infty}(M,\mathbb{R})$ satisfy PS. Let $0 \ne \alpha \in H^*(f^b,f^a)$, and $\beta \in H^*(M)$. Then:
\begin{itemize}
\item[$(i)$] $c(\alpha\cup \beta,f) \ge c(\alpha,f)$.
\item[$(ii)$] If $c(\alpha \cup \beta,f) = c = c(\alpha,f)$ for $\beta\not\in H^0(M)$, then $\beta$ is non zero in 
$$K_c := H^*(f^{-1}(c) \cap \text{Crit}(f))\, ,$$ 
and the common critical value contains infinitely many critical points. 
\end{itemize}
\end{thm}
\begin{proof}
\noindent $(i)$ Note that  $\alpha \cup \beta \in H^*(f^b,f^a)$ since relative cohomology is stable with respect to the cup product with absolute cohomology. The inequality is then obvious since, if $\alpha = 0$ in $H^*(f^{t},f^a)$, then also $\alpha \cup \beta = 0$ in $H^*(f^{t},f^a)$. \\
$(ii)$ Let $c(\alpha \cup \beta,f) = c = c(\alpha,f)$. Arguing by contradiction, suppose there exists a neighborhood $U$ of $K_c$ such that $\beta = 0$ in $H^*(U)$ for some neighbourhood $U$ of $K_c$.
Notice that --by a deformation argument (see \cite[Theorem 7.6]{Car})-- there exists a small $\varepsilon > 0$ such that $U \cup f^{c-\varepsilon}$ is a deformation retraction of $f^{c+\varepsilon}$, so that $H^*(f^{c+\varepsilon}, U \cup f^{c-\varepsilon}) = 0$. Since $\alpha=0$ in $H^*(f^{c- \varepsilon})$ and $\beta$ vanishes on $U$, $\alpha \cup \beta = 0$ in $U \cup f^{c-\varepsilon}$ and this implies $\alpha \cup \beta = 0$ in $f^{c+\varepsilon}$ and therefore $c(\alpha \cup \beta,f) \ge c(\alpha,f) + \varepsilon$. contradicting the assumption $c(\alpha \cup \beta,f) = c(\alpha,f)$.
\end{proof}
 \subsection{Spectral invariants for Lagrangians}
The aim of the this section is to apply to GFQIs the calculus of critical values introduced above. \\
\indent Let $L \subset T^*N$ be a Lagrangian generated by a GFQI $S \in C^{\infty}(N \times \mathbb{R}^k,\mathbb{R})$. Whenever we write $S^{\pm\infty}$, we mean $S^{\pm c}$ for any $c > 0$ sufficiently large. Since $N$ is compact and $S(q,\xi) = \xi^T Q \xi$ for $\vert\xi\vert > c$, by Thom isomorphism (see e.g. \cite[Theorem 6.17]{Bott-Tu}):
\begin{equation} \label{iso Thom}
H^*(S^\infty, S^{-\infty})\cong H^*(Q^\infty, Q^{-\infty})\cong H^{*-i}(N)\, ,   
\end{equation}
where $i = i(Q)$ = number of negative eigenvalues of $Q$. The calculus of critical values explained in Section \ref{PS critical} and the isomorphisms (\ref{iso Thom}) above justify the next definition. 
\begin{defn}
Let $S \in C^{\infty}(N \times \mathbb{R}^k, \mathbb{R})$ be a GFQI and $\alpha\in H^*(N) \setminus \{0\}$. Then we set
\begin{equation}
c(\alpha, S):=c(\tilde{\alpha}, S)\, ,
\end{equation}
where $\tilde{\alpha} \in H^*(S^\infty, S^{-\infty})$ is the image of $\alpha$ under the isomorphism (\ref{iso Thom}).
\end{defn}

\noindent A fundamental property of $c(\alpha, S)$ is that it is invariant with respect to the two operations on generating functions (stabilization and fibered diffeomorphisms). The next definition and proposition will be used in the proof of this fact. 
\begin{defn}
Let $(X,A)$ be a pair of topological spaces. We call $(X,A)$ compact pair if $\overline{X \setminus A}$ is compact. We call $(X,A)$ homotopy compact pair if it is properly homotopy equivalent to a compact pair.
\end{defn}

\begin{lem}
\label{aiuto Claude}
If $(X,A)$ is a homotopy compact pair then $H^*_c(X,A) \cong H^*(X,A)$. 
\end{lem} 
\begin{proof}
Assume that $(X,A)$ is properly homotopic to $(Y,B)$: $\overline{Y\setminus B}$ is compact, and $H^*_c(X,A) \cong H_c^*(Y,B)$. Let $V$ be a compact neighbourhood of $\overline{Y\setminus B}$ and $W=Y\setminus V$. Since $Y\setminus W = V$ is compact: $$H_c^*(Y,B) \cong H_c^*(Y\setminus W, B\setminus W) \cong H^*(Y\setminus W, B\setminus W)\, ,$$
by excision in cohomology with compact support (see \cite[Lemma 9, p. 320]{Spanier}. On the other hand, we obviously have $H^*(X,A) \cong H^*(Y,B)=H^*(Y\setminus W, B\setminus W)$ by standard invariance by homotopy and excision for ordinary cohomology. At the end, by using the assumption $H^*_c(X,A) \cong H_c^*(Y,B)$, we conclude that $H^*_c(X,A) \cong H^*(X,A)$.
\end{proof}

\begin{prp} \label{invariance 2 op} Let $S,S'$ be two GFQI differing by stabilization and fibered diffeomorphism and  let $\alpha\in H^*(N) \setminus \{0\}$. Then $c(\alpha, S)=c(\alpha, S')$. 
\end{prp}
\begin{proof}
\noindent This property is obvious 
$S$ and $S'$ differ for a fibered diffeomorphism, since such a diffeomorphism induces one between the sub-level sets $S^t$ and $(S')^t$. \\
\noindent Consequently, we only need to check what happens if we stabilize the GFQI $S(q;\xi)$. Without loss of generality, we may assume the nondegenerate quadratic form $\eta^T R \eta$ to be defined on a trivial bundle of rank $1$, that is $\pm \eta^2$. \\
We first consider
$$S'(q;\xi,\eta) = S(q;\xi) + \eta^2\, .$$
In such a case, $(S')^b$ [resp. $(S')^a$] projects onto $S^b$ [resp. $S^a$], the fibers being contractible. Consequently, this projection is a homotopy equivalence, so that
$H^*((S')^b, (S')^a)\cong H^*(S^b, S^a)$. \\
The argument for
$$S'(q;\xi,\eta) = S(q;\xi) - \eta^2$$
is analogous by using Alexander duality (see \cite[Theorem 16 at page 296]{spa95}) as we now explain. Set $E = N \times \mathbb{R}^k$, $\textnormal{dim}E = n+k$, then by Alexander duality:
$$H^*(S^b,S^a) \cong H_{n+k-*} (E \setminus S^a, E \setminus S^b)\, .$$
Notice that, in the previous formula, we have ignored the requirement of compact supports because we can apply Lemma \ref{aiuto Claude} to the homotopy compact pair $(S^b,S^a)$. In further details, the fact that $(S^b,S^a)$ is a homotopy compact pair can be justified as follows. 

Thanks to the PS condition, the set of critical points of $S$ in $S^b \setminus S^a$ is compact and $\varphi_X^t$, the flow of the vector field (\ref{normalized gradient}), is complete. Consequently, we can deform $S^b$ into $\tilde{S}^b$ in order that $\tilde{S}^b \setminus S^a$ contains all critical points of $S$ in $S^b \setminus S^a$ and $\overline{\tilde{S}^b \setminus S^a}$ is compact. This proves that $(S^b,S^a)$ is properly homotopy equivalent to the compact pair $(\tilde{S}^b,S^a)$. \\
Since now $E \setminus S^{t} = (-S)^{-t}$, the right hand side of the above formula can be written as:
$$H_{n+k-*} (E \setminus S^a, E \setminus S^b) = H_{n+k-*}((-S)^{-a},(-S)^{-b})\, .$$
Notice that
$$H_{n+k-*}((-S)^{-a},(-S)^{-b}) \stackrel{\stackrel{-S' = -S +\eta^2}{\uparrow}} {\cong} H_{n+k-*}((-S')^{-a},(-S')^{-b})\, ,$$
and, applying Alexander duality again,
$$H_{n+k-*}((-S')^{-a},(-S')^{-b}) \cong H^{* + 1}((S')^b,(S')^a)\, .$$
We remark that the degree shift is due to the fact that the total dimension has increased by $1$ as an effect of the given stabilization of the GFQI. In conclusion:
$$H^*(S^b,S^a) \cong  H^{* + 1}((S')^b,(S')^a)\, ,$$
which is the desired result.
\end{proof}
\noindent In what follows, we denote by $\mathfrak{L}_0(T^*N)$ the set of Lagrangians $L$ in $\mathfrak{L}(T^*N)$ which are Hamiltonian isotopic to the zero section of $T^*N$, and by ${\mathscr L}_0(T^*N)$ the set of branes over an element of $\mathfrak L_0(T^*N)$:
$$\mathscr{L}_0(T^*N) := \set{ (L, f_L): L \in \mathfrak{L}_0(T^*N) }\, .$$
The choice of a GFQI for $L \in \mathfrak{L}_0(T^*N)$ provides a primitive $f_L$ for $L$, and this is therefore equivalent to choosing a brane $(L,f_L)$. For this reason, from now on, we often indicate by $L_S$ a Lagrangian brane. \\
\noindent The previous proposition justifies the next fundamental definition. 
\begin{defn}
Let $L = L_S \in\mathscr{L}_0(T^*N)$ and $\alpha\in H^*(N) \setminus \{0\}$. We define
$$c(\alpha, L) := c(\alpha, S)$$ 
for any choice of the GFQI $S$ for $L$ up to fibered diffeomorphisms and stabilization.
\end{defn}
\begin{rem} \label{costante a}
\textnormal{
Recall that we denote $T_a(L,f_L) := (L,f_L + a)$, $a \in \mathbb{R}$. If $S$ is a GFQI associated to the brane $(L, f_L)$, the generating function giving $T_a(L, f_L)$ is simply $S+a$. Notice that $S+a$ is not quadratic at infinity but asymptotically quadratic at infinity, see Definition \ref{AQI}. Consequently --applying Lemma \ref{serve dopo 1}-- we can still define the spectral invariant for $S +a$ and we obtain that
\begin{equation*}
c(\alpha, S+a)=c(\alpha, S)+a\, .
\end{equation*}
or equivalently:
\begin{equation*}
c(\alpha, T_a(L, f_L)) = c(\alpha, L) + a
\end{equation*}
for every $a \in \R$.}
\end{rem}

\indent From now on, we assume either that the closed $n$-dimensional manifold $N$ is oriented or that cohomologies have coefficients in $\Z_2:=\Z/2\Z$. In this setting, we denote by $1 \in H^0(N)$ and $\mu \in H^n(N)$ the point class and the orientation class of $N$ respectively. We notice that the existence of $\mu \in H^n(N)$ is guaranteed by the assumptions above. 

\begin{defn} Let $L = L_{S} \in\mathscr{L}_0(T^*N)$. We define
\begin{equation*}
c_+(L) :=c (\mu, S) \quad \text{and} \quad c_-(L):=c(1, S)\, .
\end{equation*}
\end{defn}

\noindent In order to introduce the corresponding definition for a pair of Lagrangian branes, given two GFQI $S_1, S_2$, we set
$$S_1 \oplus S_2(q; \xi_1, \xi_2) := S_1(q; \xi_1) + S_2(q; \xi_2)$$
and 
$$S_1\ominus S_2(q; \xi_1, \xi_2) := S_1(q; \xi_1) - S_2(q; \xi_2)\, .$$
Notice that $S_1 \oplus S_2$ and $S_1 \ominus S_2$ are, in general, asymptotically quadratic generat-
ing functions. However, by Lemma \ref{serve dopo}, the corresponding Lagrangians admits a GFQI.
\begin{rem} \label{S1 - S2} \textnormal{Global critical points of $S_1\ominus S_2$ are in $1:1$ correspondence with intersection points between $L_{S_1}$ and $L_{S_2}$. That is
$$\text{Crit}(S_1\ominus S_2) \leftrightarrow L_{S_1} \cap L_{S_2}\, .$$
In fact, by imposing
\begin{eqnarray*} 
\begin{cases}
\frac{\partial (S_1\ominus S_2)}{\partial \xi_1}(q;\xi_1,\xi_2) & = \frac{\partial S_1}{\partial \xi_1}(q;\xi_1) = 0  \\
\frac{\partial (S_1\ominus S_2)}{\partial \xi_2}(q;\xi_1,\xi_2) & = \frac{\partial S_2}{\partial \xi_2}(q;\xi_2) = 0  \\
\frac{\partial (S_1\ominus S_2)}{\partial q}(q;\xi_1,\xi_2) & = \frac{\partial S_1}{\partial q}(q;\xi_1) - \frac{\partial S_2}{\partial q}(q;\xi_2) = 0\, ,
\end{cases}
\end{eqnarray*}
we exactly obtain 
points in $L_{S_1} \cap L_{S_2}$.}
\end{rem}
\begin{defn} \label{serve in ultima}
Let $(L_1,L_2) = (L_{S_1},L_{S_2}) \in \mathscr{L}_0(T^*N) \times \mathscr{L}_0(T^*N)$ and $\alpha\in H^*(N) \setminus \{0\}$. We define
$$c(\alpha, L_1,L_2) := c(\alpha, S_2 \ominus S_1)\, .$$   
\end{defn}
\noindent We prove here below that $c(\alpha, L_1,L_2)$ is $\mathrm{DHam}(T^*N)$-invariant. The reader should bear in mind Remark \ref{nota supp com}.
\begin{prp} \label{invariance phi} Let $(L_1,L_2) \in \mathscr{L}_0(T^*N) \times \mathscr{L}_0(T^*N)$, $\alpha\in H^*(N) \setminus \{0\}$ and $\varphi \in \mathrm{DHam}(T^*N)$. Then
$$c(\alpha,L_1,L_2) = c(\alpha,\varphi(L_1),\varphi(L_2))\, .$$
\end{prp}
\begin{proof}
Let $(\varphi_t)_{t\in[0,1]}$ be any Hamiltonian isotopy such that $\varphi_1 = \varphi$. We indicate by $S_1^t$ and $S_2^t$ the GFQIs for $\varphi_t(L_1)$ and $\varphi_t(L_2)$ respectively. Then $c(\alpha,\varphi_t(L_1),\varphi_t(L_2))$ is a critical value of $S^t_2 \ominus S^t_1$ and corresponds --see Remark \ref{S1 - S2} above-- to a point in $\varphi_t(L_1) \cap \varphi_t(L_2) = \varphi_t(L_1 \cap L_2)$. Consequently, by Sard's theorem, the map
$$t \mapsto c(\alpha, \varphi_t(L_1),\varphi(L_2))$$
has values in a set of measure zero. Since the Lagrangian isotopy $(\varphi_t)_{t \in [0,1]}$ gives continuous families of GFQIs $(S_1^t)_{t \in [0,1]}$ and $(S_2^t)_{t \in [0,1]}$, the previous map is constant, which is the desired result. 
\end{proof}
\begin{rem}\label{rem:milOh}
\textnormal{Following L. Traynor \cite{Tray}, it is possible to define the Generating Function homology of a Lagrangian. We give its definition for the reader's convenience, even though it will not  appear in these notes. Given an exact Lagrangian $L$ in $T^*N$ and two real numbers $a<b$, the Generating Function homology of $L$ in the action window $(a, b]$ in degree $k$ is defined as
\begin{equation*}
    G^{(a, b]}_k(L):=H_{k+q}(S^b, S^a)\, ,
\end{equation*} for any choice of a GFQI $S$ of $L$ with signature at infinity $q\in\N$. The proof that it is well-defined (i.e. independent from the choice of $S$) is similar to the proof of Proposition \ref{invariance 2 op}. In \cite{Viterbo1995b} (see also \cite{milOh97}) it is proved that there exists a canonical level-preserving isomorphism between this Generating Function homology and the Floer homology in the case of cotangent bundle $T^*N$.}
\end{rem}

\noindent We underline that the above result is a particular case of the so-called conformal invariance of $c(\alpha, L_1,L_2)$. We refer to \cite[Corollary 4.3]{vit92} for the proof in the case of Hamiltonian diffeomorphisms of $\R^{2n}$, which immediately extends to our setting. 
\begin{prp} \label{conformal invariance phi} Let $(L_1,L_2) \in \mathscr{L}_0(T^*N) \times \mathscr{L}_0(T^*N)$, $\alpha\in H^*(N) \setminus \{0\}$ and $\psi \in \mathrm{Diff}(T^*N)$ CES with conformal ratio $a \in \mathbb{R}_{> 0}$. Then
$$c(\alpha,\psi(L_1),\psi(L_2)) = a\,c(\psi^* \alpha,L_1,L_2)\, .$$ 
\end{prp}

\noindent Notice that, in the case where $\psi$ in the proposition above is isotopic to identity --as is the case for the flow of a conformally symplectic vector field-- then  $\psi^*\alpha=\alpha$. 

\begin{defn} Let $(L_1,L_2) = (L_{S_1}, L_{S_2}) \in \mathscr{L}_0(T^*N) \times  \mathscr{L}_0(T^*N)$. We define
\begin{equation*}
c_+(L_1, L_2) :=c (\mu, S_2 \ominus S_1)\, ,
\end{equation*}
\begin{equation*}
c_-(L_1, L_2):=c(1, S_2\ominus S_1)\, ,
\end{equation*}
\begin{equation*}
c(L_1, L_2) := \max \set{ c_+(L_1,L_2),0 } - \min \set{c_-(L_1,L_2),0 } 
\end{equation*}
and
\begin{equation*}
\gamma(L_1, L_2) :=c_+(L_1,L_2) - c_-(L_1, L_2)\, .
\end{equation*}
\end{defn}
\begin{rem}
\textnormal{Let $((L_1, f_{L_1}),(L_2,f_{L_2})) \in \mathscr{L}_0(T^*N) \times \mathscr{L}_0(T^*N)$. With the same argument of Remark \ref{costante a}, it is easy to see that
$$c_\pm((L_1,f_{L_1}),T_a(L_2,f_{L_2})) = c_\pm((L_1,f_{L_1}),(L_2,f_{L_2})) + a\,$$
and
$$c_\pm(T_a(L_1,f_{L_1}),(L_2,f_{L_2})) = c_\pm((L_1,f_{L_1}),(L_2,f_{L_2})) - a\,.$$}
\end{rem}
\indent The rest of this section is devoted to proving that $c$ and $\gamma$ are distances, respectively on $\mathscr{L}_0(T^*N)$ and $\mathfrak{L}_0(T^*N)$. For the proof of this fundamental fact, we shall require two lemmata . 
\begin{lem} \label{mag zero} Let $(L_1,L_2) \in \mathscr{L}_0(T^*N) \times  \mathscr{L}_0(T^*N)$. Then the next equality holds:
\begin{equation} \label{inf a}
\gamma(L_1,L_2) = \inf \set{ c(T_a(L_1,f_{L_1}),(L_2,f_{L_2})): a \in \mathbb{R}}\, .
\end{equation}
\end{lem}
\begin{proof} 
We first notice that we may change $f_{L_1}$ into $f_{L_1} + a$ for any constant $a \in \mathbb{R}$. In this way, $c_{\pm}(L_1,L_2)$ is changed into $c_{\pm}(L_1,L_2) - a$ and $c(L_1,L_2)$ is changed into 
$$\max\{ c_+(L_1,L_2) - a,0\} -\min \set{ c_-(L_1,L_2) - a,0 }\, .$$
The last expression is minimal when $a = c_+(L_1,L_2)$ and then takes the value
$$c_+(L_1,L_2) - c_-(L_1,L_2) = \gamma(L_1,L_2)\, .$$
\end{proof}
\noindent For the proof of the next lemma, we refer to \cite[Proposition 3.3]{vit92}.
\begin{lem} \label{triangle}
Let $S_1, S_2$ be two GFQIs. For any $\alpha, \beta$ in $H^*(N)$, it holds:
\begin{equation*}
c(\alpha\cup\beta, S_1 \oplus S_2) \geq c(\alpha, S_1) +c (\beta, S_2)\, .
    \end{equation*}
\end{lem}
\begin{thm} \label{distanze}
The function $c$ is a distance on $\mathscr{L}_0(T^*N)$ and the function $\gamma$ is a distance on $\mathfrak{L}_0(T^*N)$. Moreover, $\mathrm{DHam}_c(T^*N)$ acts by isometries on $(\mathscr{L}_0(T^*N),c)$ and $(\mathfrak{L}_0(T^*N),\gamma)$.
\end{thm}
\begin{proof} We start by noting that it is sufficient to give the proof for $c$. By Lemma \ref{mag zero}, proving that $c$ is a distance on $\mathscr{L}_0(T^*N)$, immediately implies the analogous statement for $\gamma$. \\
Let $(L_1,L_2) \in \mathscr{L}_0(T^*N) \times  \mathscr{L}_0(T^*N)$. \\
-- Positivity: $c((L_1, f_{L_1}) , (L_2, f_{L_2})) = 0 \Leftrightarrow (L_1, f_{L_1}) = (L_2, f_{L_2})$. \\
Since $\mu \cup 1 = \mu$, applying Theorem \ref{thm:lusSch} we get:
$$c_+(L_1,L_2) = c(\mu,S_2 \ominus S_1) = c(\mu \cup 1,S_2 \ominus S_1) \ge c(1,S_2 \ominus S_1) = c_-(L_1,L_2)\, .$$
It is then sufficient to notice that
$$0 \le c_+(L_1,L_2) - c_-(L_1,L_2) \le \max \{ c_+(L_1,L_2),0\} - \min \{c_-(L_1,L_2),0\}\, ,$$
that is
$$0 \le \gamma(L_1,L_2) \le c(L_1,L_2)\, ,$$
hence $c(L_1,L_2) \ge 0$. \\
Moreover, we observe that $\gamma(L_1,L_2) = 0$ if and only if $L_1 = L_2$. By definition:
$$\gamma(L_1,L_2) = 0 \Leftrightarrow c(1, S_2 \ominus S_1) = c(\mu, S_2 \ominus S_1) = c(\mu \cup 1, S_2 \ominus S_1)\, .$$
By Lusternik-Schnirelmann theorem in the equality case (see $(ii)$ of Theorem \ref{thm:lusSch}), we have that $\mu \ne 0$ in $H^*(\text{Crit}(S_2 \ominus S_1))$. Equivalently --by Remark \ref{S1 - S2}-- $\mu \ne 0$ in $H^*(L_1 \cap L_2)$. Since $\mu$ is a $n$-dimensional cohomology class, this implies that $L_1 = L_1 \cap L_2 = L_2$ and therefore $L_1 = L_2$. The positivity of $\gamma$ implies also the positivity of $c$. Indeed, since $0 \le \gamma(L_1,L_2) \le c((L_1, f_{L_1}),(L_2, f_{L_2}))$, 
$$c((L_1,f_{L_1}),(L_2, f_{L_2})) = 0 \Rightarrow \gamma(L_1,L_2) = 0 \Rightarrow L_1 = L_2\, ,$$
so that 
$$(L_2,f_{L_2}) = T_a (L_1,f_{L_1}) = (L_1,f_{L_1}+a)\, ,$$
for a constant $a \in \mathbb{R}$. This implies: 
$$0 = c((L_1,f_{L_1}),(L_1,f_{L_1}+a)) = a\, .$$
Hence $a = 0$, which means $(L_1, f_{L_1}) = (L_2, f_{L_2})$. \\
-- Symmetry: $c(L_1,L_2) = c(L_2,L_1)$. \\
In order to prove the symmetry, we need to recall that, as a consequence of Alexander duality (see \cite[Corollary 2.8.]{vit92}):
$$c_+(L_1,L_2):= c(\mu, S_2 \ominus S_1) = -c(1,S_1 \ominus S_2) =: -c_-(L_2,L_1).$$
From the above equality, symmetry for $c$ immediately follows, once reminded the elementary identity:
\begin{equation} \label{min -max}
    -\min\Set{a, 0}=\max\Set{-a, 0}\, .
\end{equation}
-- Triangle inequality: $c(L_1,L_2) \le c(L_1,L_2) + c(L_2,L_3)$. \\
This follows from Lemma \ref{triangle} above and \cite[Corollary 2.8.]{vit92}, as they imply the inequality:
$$c(1,S_1 \ominus S_3) \ge c(1,S_1) + c(1,-S_3) = c(1,S_1) - c(\mu,S_3)\, ,$$
that is
\begin{equation} \label{prima}
c(\mu,S_3 \ominus S_1) \le c(\mu,S_3) - c(1,S_1)\, . 
\end{equation}
By the same arguments we obtain the inequality
$$c(1,S_3 \ominus S_1) \ge c(1,S_3) + c(1,-S_1) = c(1,S_3) - c(\mu,S_1)\, ,$$
that is
\begin{equation} \label{seconda}
-c(1,S_3 \ominus S_1) \le c(\mu,S_1) - c(1,S_3)\, . 
\end{equation}
\noindent Let now $(L_1, f_{L_1}), (L_2, f_{L_2}), (L_3, f_{L_3}) \in \mathscr{L}_0(T^*N)$ (we omit in the sequel the choice of primitive $f_{L_i}$).
First notice that, by (\ref{prima}) and (\ref{seconda}), we have that
\begin{eqnarray} \label{pre}
c(L_1,L_3) := \max \set{c(\mu,S_3\ominus S_1),0} -\min \set{c(1,S_3\ominus S_1),0} \nonumber \\
\le \max \set{c(\mu,S_3) -c(1,S_1), 0} -\min \set{c(1,S_3)-c(\mu,S_1),0}\, .
\end{eqnarray}
Using (\ref{min -max}), we immediately obtain 
\begin{equation}\label{eq:triangIneqC1}
    \max\Set{c(\mu, S_3)-c(1, S_1), 0}\leq \max\Set{c(\mu, S_3), 0}-\min\Set{c(1, S_1), 0}\, .
\end{equation} 
Moreover, exchanging $S_1$ and $S_3$ in the previous inequality and using (\ref{min -max}) in the left-hand side of the equation, we  have 
\begin{equation}\label{eq:triangIneqC2}
-\min\Set{ c(1, S_3) - c(\mu, S_1), 0}\leq \max\Set{c(\mu, S_1), 0}-\min\Set{c(1, S_3), 0}\, .
\end{equation}
Since $c$ is invariant by Hamiltonian action (see Proposition \ref{invariance phi}), we can suppose that $L_2 = \mathcal{O}_N$. The, triangular inequality for $c$ can then be restated as
\begin{eqnarray*}
c(L_1,L_3) := \max \set{c(\mu,S_3\ominus S_1),0} -\min \set{c(1,S_3\ominus S_1),0} \le c(L_1,\mathcal{O}_N) + c(\mathcal{O}_N,L_3) \\
:= -\min \set{ c(1,S_1),0} + \max \set{ c(\mu,S_1),0} + \max \set{c(\mu,S_3),0} - \min \set{c(1,S_3),0}\, .
\end{eqnarray*}
Its validity is then a straightforward consequence of inequalities (\ref{pre}), (\ref{eq:triangIneqC1}) and (\ref{eq:triangIneqC2}). \\
-- The fact that $\mathrm{DHam}(T^*N)$ acts by isometries on $(\mathscr{L}_0(T^*N),c)$ and $(\mathfrak{L}_0(T^*N),\gamma)$ is a straightforward consequence of Proposition \ref{invariance phi}.
\end{proof}
\begin{es} \textnormal{Let $f$ be a smooth function on a closed connected manifold $N$. If
$$L = \Gamma_f = \set{ (q,df(q)): q \in N}\, ,$$
then
$$c(\mu,\Gamma_f) = \max_{x \in M}f(x) \quad \text{and} \quad c(1,\Gamma_f) = \min_{x \in M}f(x)\, .$$
Consequently, if $g \in C^{\infty}(N,\mathbb{R})$, the next identities are verified:
$$\gamma(\Gamma_f,\Gamma_g) = \max_{x \in N} (g-f)(x) - \min_{x \in N} (g-f)(x) =: \mathrm{osc}(g-f) \le 2 \norm{g-f}_{C^0}$$
and
$$c(\Gamma_f,\Gamma_g) =  \max\set{ \mathrm{osc}(g-f), \norm{g-f}_{C^0} }\, .$$
}
\end{es}
\noindent We conclude this section by recalling the relation between the distance $\gamma$ and the Hofer norm $\norm{H}$. Let $L \in \mathfrak{L}_0(T^*N)$ and $\varphi \in \mathrm{DHam}_c(T^*N)$, given by the Hamiltonian $H: [0,1] \times T^*N \to \mathbb{R}$. Then (see \cite{vit92}):
\begin{equation}
\gamma(L,\varphi(L)) \le \norm{H} := \int^1_0 \left( \max_{z \in T^*N} H(t,z) - \min_{z \in T^*N} H(t,z) \right) dt\, .
\end{equation}
\section{Completions, \texorpdfstring{$\gamma$-support and $\gamma$-coisotropy}{gamma-support and gamma-coisotropy}} \label{sezione 4}
\begin{center}
\begin{minipage}{10 cm}
\textit{Content of Section \ref{sezione 4}}. We introduce the completions $\widehat{\mathfrak{L}_0}(T^*N)$ and $\widehat{\mathscr{L}_0}(T^*N)$ of the metric spaces defined in Section \ref{SPEC}. Their elements are abstract objects but they admit a geometric counterpart, first introduced by  V. Humilière in \cite{hum08} and more recently exploited and called ``$\gamma$-support'' by C. Viterbo. To understand the elements in the completions, we introduce the definitions of $c$-support, $\gamma$-support and $\gamma$-coisotropy. After presenting some instructive examples, we prove that the notion of $\gamma$-coisotropic subset generalizes that of coisotropic subset and that $\gamma$-supports are $\gamma$-coisotropic. Finally, we briefly discuss some recent questions about the study of $\gamma$-supports.
\end{minipage}
\end{center}

\subsection{\texorpdfstring{$c$-support}{c-support}, \texorpdfstring{$\gamma$-support and $\gamma$-coisotropy}{gamma-support and gamma-coisotropy}}

\indent The metric spaces $(\mathfrak{L}_0(T^*N),\gamma)$ and $(\mathscr{L}_0(T^*N),c)$ defined in Section \ref{SPEC} are not complete (and not even Polish, see \cite[Proposition A.1.]{vit22}. We denote their respective completions by $\widehat{\mathfrak{L}_0}(T^*N)$ and $\widehat{\mathscr{L}_0}(T^*N)$. Their study was initiated in \cite{hum08} (in their Hamiltonian setting in $\mathbb{R}^{2n}$) and pushed further in \cite{vit22}. As a consequence of Theorem \ref{distanze}, the group of compactly supported Hamiltonian diffeomorphisms acts by isometries on $\widehat{\mathfrak{L}_0}(T^*N)$ and $\widehat{\mathscr{L}_0}(T^*N)$. 
\noindent In what follows, with an abuse of notation, we indicate by $L$ elements of $\widehat{\mathfrak{L}_0}(T^*N)$ and by $\tilde L$ elements in $\widehat{\mathscr{L}_0}(T^*N)$, omitting  the choice of primitive for a given brane for the sake of readability.
\begin{rem}
    \textnormal{Before proceeding, let us remark that Lemma \ref{mag zero} swiftly implies that the projection $\mathscr{L}_0(T^*N)\rightarrow\mathfrak{L}_0(T^*N)$ is Lipschitz with respect to the distances $c$ and $\gamma$, and in particular extends to a continuous map between the two completions.}  
\end{rem}
\begin{defn} Let $z\in T^*N$. \\
$(i)$ Let $\tilde{L} \in\widehat{\mathscr{L}_0}(T^*N)$. We say that $z\in c \textnormal{-supp}(\tilde{L})$ if and only if for every $\varepsilon>0$ there exists $\varphi\in\mathrm{DHam}_c(T^*N)$, supported in $B(z, \varepsilon)$, such that
    \begin{equation}
        c(\tilde{L}, \varphi (\tilde{L}))>0\, .
    \end{equation}
$(ii)$ Similarly, let $L \in \widehat{\mathfrak{L}_0}(T^*N)$. We say that $z \in \gamma\textnormal{-\supp}(L)$ if and only if for every $\varepsilon>0$ there exists $\varphi\in\mathrm{DHam}_c(T^*N)$, supported in $B(z, \varepsilon)$, such that
    \begin{equation} \label{dis completato}
        \gamma(L, \varphi (L))>0\, .
    \end{equation}
\end{defn}
\noindent We remark that condition (\ref{dis completato}) is equivalent to $L \neq \varphi(L)$ in $\widehat{\mathfrak{L}_0}(T^*N)$
\noindent Before going further, we list some properties of supports. 
\begin{rem} \label{3remarks} \textnormal{$(i)$ If $\tilde{L} \in \mathscr{L}_0(T^*N)$ [resp. $L \in \mathfrak{L}_0(T^*N)$], then $c\textnormal{-supp}(\tilde{L}) = \tilde{L}$ [resp. $\gamma\textnormal{-supp}(L) = L$]. \\
$(ii)$ Let $\tilde{L} \in \widehat{\mathscr{L}_0}(T^*N)$ [resp. $L \in \widehat{\mathfrak{L}_0}(T^*N)$]. The definition automatically implies that $c\textnormal{-supp}(\tilde{L})$ [resp. $\gamma\textnormal{-supp}(L)$] is closed, even though it may be very complicated. \\
$(iii)$ Let $\tilde{L} \in \widehat{\mathscr{L}_0}(T^*N)$ [resp. $L \in \widehat{\mathfrak{L}_0}(T^*N)$]. If $\psi \in \mathrm{Diff}(T^*N)$ is CES, then $c\textnormal{-supp}(\psi(\tilde{L})) = \psi(c\textnormal{-supp}(\tilde{L}))$ [resp. $\gamma\textnormal{-supp}(\psi(L)) = \psi(\gamma\textnormal{-supp}(L))$]. We refer to \cite[Corollary 6.22.]{vit22} or leave it as an exercise for the interested reader}.
\end{rem}
\indent Let $L \in \mathfrak{L}_0(T^*N)$ be the Lagrangian submanifold associated to a brane $\tilde{L} \in \mathscr{L}_0(T^*N)$. Notice that we may have $c(\tilde{L},\varphi(\tilde{L})) > 0$ but $\gamma(L,\varphi(L)) = 0$. For example, for $a \ne 0$, let $H$ be a compactly supported Hamiltonian, $(\varphi_t)_{t \in [0,1]}$ be its Hamiltonian flow and $\varphi = \varphi_1$ such that
$$\varphi(\tilde{L}) = T_a\tilde{L}\, .$$
This is the case e.g. if $H\equiv a$ on $L$ (see the proof of Proposition \ref{no empty}).  Then we may check that
$$c_+(\tilde{L}, \varphi(\tilde{L})) = a = c_-(\tilde{L}, \varphi(\tilde{L}))$$
so that
$$c(\tilde{L},\varphi(\tilde{L})) = |a|  $$ while $$ \gamma(L, \varphi(L)) = 0\, .$$

\noindent With the next proposition, we show that --in order to study supports-- we can get rid of the $c$-support.   
\begin{prp}
Let $L \in \widehat{\mathfrak{L}_0}(T^*N)$ corresponding to $\tilde{L} \in \widehat{\mathscr{L}_0}(T^*N)$. Then
\begin{equation}
c\textnormal{-\supp}(\tilde{L})=\gamma\textnormal{-\supp}(L)\, .
\end{equation}
\end{prp}
\begin{proof} Let $z \in T^*N$, $\varepsilon >0$ and $\varphi \in \mathrm{DHam}_c(T^*N)$ supported in $B(z,\varepsilon)$. Clearly 
$$\varphi(L) \ne L \Rightarrow \varphi(\tilde{L}) \ne \tilde{L}\, ,$$
so that the inclusion $\gamma\textnormal{-supp}(L) \subseteq c\textnormal{-supp}(\tilde{L})$ is satisfied. In order to prove the other one, we want to prove that the case $\varphi(L) = L$ and $\varphi(\tilde{L}) \ne \tilde{L}$ is impossible for $\varphi$ supported in a small ball. Let $z \notin \gamma\textnormal{-supp}(L)$. This means that there exists $\varepsilon > 0$ such that for all $\varphi \in \mathrm{DHam}_c(T^*N)$, supported in $B(z,\varepsilon)$, we have $\varphi(L) = L$. Equivalently, $\varphi(\tilde{L}) = T_a \tilde{L}$ for some $a \in \mathbb{R}$. To conclude, we need to prove that $a = 0$. Let us then consider any $f \in C^{\infty}(N,\mathbb{R})$ such that $\Gamma_f \cap B(z,\varepsilon) = \emptyset$. Then
$$c_+(\varphi(\Gamma_f),\varphi(\tilde L)) = c_+(\varphi(\Gamma_f),T_a\tilde{L}) = c_+(\varphi(\Gamma_f),\tilde{L}) + a = c_+(\Gamma_f,\tilde{L}) + a\, .$$
Since, by Proposition \ref{invariance phi}, $c_+(\varphi(\Gamma_f),\varphi(\tilde L)) = c_+(\Gamma_f,\tilde{L})$, we immediately conclude that $a = 0$ so that $\varphi(\tilde{L}) = \tilde{L}$. 
\end{proof}
\indent We now examine the relationship between the $\gamma$-support and the support of an isotopy (the latter notion is recalled below). The key ingredient is the following fragmentation lemma due to A. Banyaga \cite[Lemma~III.3.2]{ban78}.

\begin{lem}\label{Banyaga}
Let $(M,\omega)$ be a closed symplectic 
manifold and $(U_j)_{j \in [1,N]}$ be an open cover of $M$. Then any Hamiltonian isotopy $(\varphi_t)_{t \in [0,1]}$ can be written as a 
composition of Hamiltonian isotopies $(\varphi_{j,t})_{t \in [0,1]}$ with Hamiltonian supported in some $U_{k(j)}$. The same holds for $M$ open and compactly supported Hamiltonian isotopies.
\end{lem}
\noindent By ``support of an isotopy $(\varphi_t)_{t \in [0,1]}$'', we mean the closure of the set $$\set{ z \in M: \text{ there exists } t \in (0,1] \text{ such that } \varphi_t(z) \ne z }.$$ 
If the complement of the support is connected and the isotopy is generated by a Hamiltonian $H$ defined on $[0,1] \times M$, this is also the projection on $M$ of the support of $H$. When
the isotopy is implicit, we write $\textnormal{supp}(\varphi)$ for the support of the corresponding isotopy such that $\varphi_1 = \varphi$. Note that in Banyaga’s fragmentation theorem, if the open cover is by small balls, we may assume the support of the Hamiltonians are in the $U_{k(j)}$ (since the complement of a small ball is always connected). \\
We shall often use the next proposition.
\begin{prp} \label{serve dopo}
Let $L \in\widehat{\mathfrak{L}_0}(T^*N)$. Let $\varphi\in \mathrm{DHam}_c(T^*N)$ such that $\varphi (L) \ne L$. Then
$\gamma\textnormal{-supp}(L)\cap \textnormal{supp}(\varphi)\neq\emptyset$.
\end{prp}
\begin{proof} Assume, by contradiction, that $\gamma\textnormal{-supp}(L) \cap \textnormal{supp}(\varphi) = \emptyset$. Let $z \in \textnormal{supp}(\varphi)$ so that $z \notin \gamma\textnormal{-supp}(L)$. This means that there exists $\varepsilon = \varepsilon(z) > 0$ such that for every $\psi \in \mathrm{DHam}_c(T^*N)$ supported in $B(z,\varepsilon)$ we have $\psi(L) = L$. Take a finite cover $(B(z_j,\varepsilon_j))_{j \in [1,N]}$ of $\textnormal{supp}(\varphi)$ given by these balls. By Lemma \ref{Banyaga}, we can write
$$\varphi = \psi_1 \circ \ldots \circ \psi_N\, ,$$
where the $\psi_j \in \mathrm{DHam}_c(T^*N)$ are supported in some $B(z_{k(j)},\varepsilon_{k(j)})$ and such that $\psi_j(L) = L$. This implies $\varphi(L) = L$, giving the desired contradiction. 
\end{proof}
\indent We now introduce the definition of $\gamma$-coisotropic subset in a symplectic manifold. Of course we will prove that this notion coincides with usual coisotropy in  case the subset is a smooth submanifold ( see the beginning of Section \ref{sezione 23}). Note that an analogue of $\gamma$-coisotropy, with $\gamma$ replaced by the Hofer distance (going under the name of “local rigidity”) had been defined by Usher in \cite{ush19}. \\
In what follows, the $\gamma$-norm of $\varphi \in \mathrm{DHam}(T^*N)$ is defined as
\begin{equation}
\gamma(\varphi) : =\sup_{L\in\mathfrak{L}_0(T^*N)} \gamma(L, \varphi(L))\, .
\end{equation}
We refer to \cite[Definition 2.1 and Proposition 2.3]{carVit08} for the detailed properties of $\gamma$. Moreover, we have:
$$\gamma(\varphi) \le \text{osc}(H) := \max_{(z,t) \in T^*N \times [0,1]} H(t,z) - \min_{(z,t) \in T^*N \times [0,1]} H(t,z) \, ,$$
see \cite[Proposition 2.6]{carVit08} for the proof. \\
Similarly to $\widehat{\mathfrak{L}_0}(T^*N)$ and $\widehat{\mathscr{L}_0}(T^*N)$, $\widehat{\mathrm{DHam}_c}(T^*N)$ and $\widehat{\mathrm{DHam}}(T^*N)$ denotes respectively the completion of $(\mathrm{DHam}_c(T^*N),\gamma)$ and $(\mathrm{DHam}(T^*N),\gamma)$. We have that $\textnormal{Homeo}(T^*N) \subseteq \widehat{\mathrm{DHam}}(T^*N)$.
\begin{defn}\label{defn:gamCoi}
    Let $V$ be a subset of $T^*N$. \\
$(i)$ We say that $V$ is $\gamma$-coisotropic at $z\in V$ if there exists $\varepsilon > 0$ such that for any ball $B(z,\eta)$ with $0<\eta<\varepsilon$ there exists $\delta = \delta(\eta)>0$ such that for all $\varphi \in \mathrm{DHam}_c(T^*N)$, supported in $B(z,\varepsilon)$ and such that $\varphi(V) \cap B(z,\eta) = \emptyset$, we have that $\gamma(\varphi)>\delta$. See Figure \ref{fig:gamma coisotropic}. \\
$(ii)$ We say that $V \subset T^*N$ is $\gamma$-coisotropic if $V \ne \emptyset$ and $V$ is $\gamma$-coisotropic at each $z \in V$. 
\end{defn}
\begin{figure}[ht]
  \centering
    %% Creator: Inkscape 1.2.2 (b0a8486541, 2022-12-01), www.inkscape.org
%% Accompanies image file 'gammaCoisotropic.pdf' (pdf, eps, ps)
\begingroup%
  \makeatletter%
  \providecommand\color[2][]{%
    \errmessage{(Inkscape) Color is used for the text in Inkscape, but the package 'color.sty' is not loaded}%
    \renewcommand\color[2][]{}%
  }%
  \providecommand\transparent[1]{%
    \errmessage{(Inkscape) Transparency is used (non-zero) for the text in Inkscape, but the package 'transparent.sty' is not loaded}%
    \renewcommand\transparent[1]{}%
  }%
  \providecommand\rotatebox[2]{#2}%
  \newcommand*\fsize{\dimexpr\f@size pt\relax}%
  \newcommand*\lineheight[1]{\fontsize{\fsize}{#1\fsize}\selectfont}%
  \ifx\svgwidth\undefined%
    \setlength{\unitlength}{255.11808861bp}%
    \ifx\svgscale\undefined%
      \relax%
    \else%
      \setlength{\unitlength}{\unitlength * \real{\svgscale}}%
    \fi%
  \else%
    \setlength{\unitlength}{\svgwidth}%
  \fi%
  \global\let\svgwidth\undefined%
  \global\let\svgscale\undefined%
  \makeatother%
  \begin{picture}(1,0.55555565)%
    \lineheight{1}%
    \setlength\tabcolsep{0pt}%
    \put(0,0){\includegraphics[width=\unitlength,page=1]{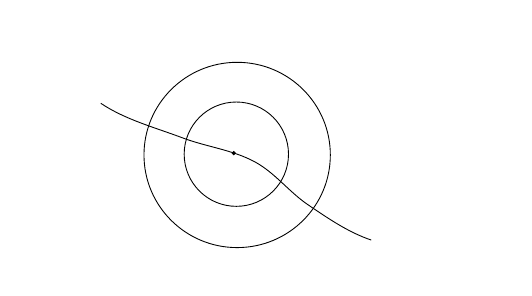}}%
    \put(0.41671371,0.23309479){\color[rgb]{0,0,0}\makebox(0,0)[lt]{\lineheight{1.25}\smash{\begin{tabular}[t]{l}$z$\end{tabular}}}}%
    \put(0,0){\includegraphics[width=\unitlength,page=2]{gammaCoisotropic.pdf}}%
    \put(0.464716,0.31276337){\color[rgb]{0,0,0}\makebox(0,0)[lt]{\lineheight{1.25}\smash{\begin{tabular}[t]{l}$\eta$\end{tabular}}}}%
    \put(0,0){\includegraphics[width=\unitlength,page=3]{gammaCoisotropic.pdf}}%
    \put(0.36056488,0.1546461){\color[rgb]{0,0,0}\makebox(0,0)[lt]{\lineheight{1.25}\smash{\begin{tabular}[t]{l}$\varepsilon$\end{tabular}}}}%
    \put(0,0){\includegraphics[width=\unitlength,page=4]{gammaCoisotropic.pdf}}%
    \put(0.6074761,0.08479984){\color[rgb]{0,0,0}\makebox(0,0)[lt]{\lineheight{1.25}\smash{\begin{tabular}[t]{l}$V$\end{tabular}}}}%
    \put(0.61389267,0.33695298){\color[rgb]{0,0,0}\makebox(0,0)[lt]{\lineheight{1.25}\smash{\begin{tabular}[t]{l}$\varphi(V)$\end{tabular}}}}%
  \end{picture}%
\endgroup%

  \caption{The set $V$ is $\gamma$-coisotropic at $z \in V$ if $\gamma(\varphi)>\delta$.}
\label{fig:gamma coisotropic}
\end{figure}
\noindent In simple terms, the set $V$ is $\gamma$-coisotropic at $z \in V$ when --if you want to move $V$ outside any small ball centered in $z \in V$ of radius $\eta > 0$-- then you need a minimal ``amount of energy''.
\begin{es} \textnormal{The origin $\{(0,0)\} \subset\R^2$ is not $\gamma$-coisotropic. Let us consider a Hamiltonian $H$ compactly supported near the origin and such that $\partial_pH(q,0) = 1$ for $(q,0)$ contained in the support of $H$. Clearly, $\textnormal{osc}(H)$ can be arbitrary small. By using the inequality $\gamma(\varphi) \le \textnormal{osc}(H)$, the corresponding time-1 flow $\varphi$ displaces the origin from a small ball and has arbitrarily small $\gamma$-norm.} 
\end{es}
In order to prove that the notion of $\gamma$-coisotropic subset generalizes that of coisotropic subset, we first show that every Lagrangian is $\gamma$-coisotropic.
\begin{lem} \label{lem Lag} Every $L \in \mathfrak{L}(T^*N)$ is $\gamma$-coisotropic.
\end{lem}
\begin{proof}
The property of being $\gamma$-coisotropic is a local one, i.e. a set is $\gamma$-coisotropic if and only if a sufficiently small neighborhood of each of its points is $\gamma$-coisotropic. This allows us to reduce ourselves to consider a Weinstein chart around a point $z\in L$, containing a small $B_{2n}(z, \varepsilon)$. Through such a chart, $L$ is locally identified with 
$$(\R^n\times \{0\}) \cap B_{2n}(0, \varepsilon)\subset \R^{2n}\, .$$ 
and in the following arguments we tacitly assume we remain in this chart. \\
Fix any positive $\eta \in (0,\varepsilon)$, and consider the ball $B_{2n}(0, \eta)$. We need to show that if $\varphi\in \mathrm{DHam}_c(\R^{2n})$ is supported in $B_{2n}(0, \varepsilon)$ and $\varphi(L)\cap B_{2n}(0, \eta)=\emptyset$, then $\gamma(\varphi)>\delta$, for some $\delta = \delta(\eta) > 0$. \\
We start by constructing an auxiliary exact Lagrangian $L_1$ such that $L$ and $L_1$ intersect transversely at the origin, have no other transverse intersection points, coincide outside of $B_{2n}(0, \eta)$ and 
$\gamma(L, L_1)\geq  C \eta^2$ for a real constant $C>0$. \\
The construction of $L_1$ is simple. $L_1$ is locally identified with the graph of $df$ where $f$ is a smooth function supported in $B_{n}(0, \eta)$, with exactly two critical points, the maximum of order $\eta^2$ and the minimum equal to $0$. Since $L \ne L_1$, we must have
$$\gamma(L, L_1)\geq C\eta^2\, .$$
Let $L_t$ be the graph of $tdf$. We now argue as in the proof of Proposition \ref{invariance phi}. Since $\varphi(L) \cap L_t$ does not depend on $t$, the quantity $\gamma(\varphi(L), L_t)$ takes values in a fixed set of measure zero, therefore is also constant. As a result, we have that 
$$\gamma(\varphi(L), L_1)=\gamma(\varphi(L), L)\, .$$
\begin{figure}[ht]
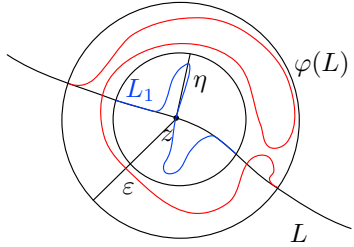

  \centering
    
  \caption{The shape of the Lagrangian $L_1$ involved in the proof.}
  \label{fig:limLagProof}
\end{figure}
\noindent The triangle inequality for $\gamma$ gives $ \gamma(L, L_1)\leq \gamma(L, \varphi(L))+\gamma(\varphi(L), L_1)
$. Consequently, by using the previous equality and the symmetry of $\gamma$, we obtain:
    \begin{equation*}
        \gamma(L, \varphi(L)) \ge \frac{\gamma(L,L_1)}{2} \geq  C\eta^2\, .
    \end{equation*}
This proves that $L$ is $\gamma$-coisotropic at every point $z \in L$ so that $L$ is $\gamma$-coisotropic.
\end{proof}
\begin{es} \textnormal{Let $V := [0,1] \times \{0\} \subset \R^2$. Arguing as in the previous example, we conclude that $V$ is not $\gamma$-coisotropic at $(0,0)$ and $(1,0)$. However $V$ is $\gamma$-coisotropic at every point $z := (x,0)$, with $x \in (0,1)$. In order to prove this fact, it is sufficient to notice that coisotropy is a local property and that $V$ coincides with a closed Lagrangian at each interior point. Hence $V := [0,1] \times \{0\}$ is not $\gamma$-coisotropic.} 
\end{es}
\begin{prp} \label{prp:coiImpGamCoi} Let $V$ be a smooth, closed submanifold of $T^*N$. Then $V$ is $\gamma$-coisotropic if and only if $V$ is coisotropic.
\end{prp}
\begin{proof} We only prove that if $V$ is coisotropic then $V$ is $\gamma$-coisotropic. We refer to \cite[Proposition 7.6.]{vit22} for details of the other direction. If $V$ is coisotropic, locally $V$ can be identified with
$$\set{ (q_1, \ldots, q_n,p_1,\ldots,p_k,0,\ldots,0): q_j,p_j \in \mathbb{R}}\, .$$
This contains the Lagrangian $\set{ (q_1, \ldots, q_n,0,\ldots,0): q_j \in \mathbb{R}}$ which, by Lemma \ref{lem Lag} is $\gamma$-coisotropic. Then $V$ is $\gamma$-coisotropic.
\end{proof}
\indent We aim now to give a proof that $\gamma$-supports are $\gamma$-coisotropic. In order to do so, we need  the fact that the $\gamma$-support of the $\gamma$-limit of a sequence $(L_n)_{n \ge 1}$ in $\widehat{\mathfrak{L}_0}(T^*N)$ is contained in the topological lower limit of the sequence $(\gamma\textnormal{-supp}(L_n))_{n \ge 1}$.
\begin{lem}\label{lem:gamSupLimInf}
Let $(L_n)_{n \ge 1}$ be a sequence in $\widehat{\mathfrak{L}_0}(T^*N)$ $\gamma$-converging to $L \in \widehat{\mathfrak{L}_0}(T^*N)$. Then
$$\gamma\textnormal{-supp}(L) \subset \liminf \gamma\textnormal{-supp}(L_n):=\{x=\lim_n x_n \mid x_n \in \gamma\textnormal{-supp}(L_n)\} .$$
In particular, if there exists $z \in T^*N$ and $\varepsilon > 0$ such that $\gamma\textnormal{-supp}(L_n) \cap B(z, \varepsilon) = \emptyset$ for all $n \gg 0$, then $\gamma\textnormal{-supp}(L) \cap B(z, \varepsilon)=\emptyset$.
\end{lem}

\begin{proof}The latter statement is a simple consequence of the former one. We argue by contradiction. If $\gamma\textnormal{-supp}(L) \cap B(z, \varepsilon)\ne\emptyset$, we may assume the center of the ball $B(z,\varepsilon)$ to be in $\gamma\textnormal{-supp}(L)$. Consequently, there exists $\varphi \in \mathrm{DHam}_c(T^*N)$, supported in $B(z,\varepsilon)$, such that $\gamma(L,\varphi(L)) > 0$. 
Then, since by hypothesis $\gamma(L,\varphi(L)) = \lim_{n \to +\infty} \gamma(L_n,\varphi(L_n))$, we have $\gamma(L_n,\varphi(L_n)) > 0$ for $n \gg 0$. This implies --by Proposition \ref{serve dopo}-- that $$\textnormal{supp}(\varphi) \cap \gamma\textnormal{-supp}(L_n) \ne \emptyset \quad \text{for} \quad n \gg 0\, .$$ 
Since $\textnormal{supp}(\varphi) \subset B(z,\varepsilon)$, we therefore have that 
$$B(z,\varepsilon) \cap \gamma\textnormal{-supp}(L_n) \ne \emptyset\quad \text{for} \quad n \gg 0\, ,$$  
which gives the desired contradiction. 
\end{proof}
\noindent
From the previous result, the next proposition easily follows. 
\begin{prp} \label{gamma coisotropic}
If $L \in \widehat{\mathfrak{L}_0}(T^*N)$ then $\gamma\textnormal{-supp}(L)$ is $\gamma$-coisotropic.
\end{prp}
\begin{proof}
Let $z \in \gamma\textnormal{-supp}(L)$ and assume --by contradiction-- that $\gamma\textnormal{-supp}(L)$ is not $\gamma$-coisotropic at $z$. By definition, this means that for every $\varepsilon>0$ there exist $\eta \in (0, \varepsilon)$ and a sequence $(\varphi_n)_{n \ge 1}$ in $\mathrm{DHam}(T^*N)$, supported in $B(z,\varepsilon)$, such that
$$\gamma(\varphi_n) \to 0 \quad \text{and} \quad \varphi_n(\gamma\textnormal{-supp}(L)) \cap B(z,\eta) = \emptyset\, .$$
Recall that --see $(iii)$ of Remark \ref{3remarks}-- $\varphi_n(\gamma\textnormal{-supp}(L)) = \gamma\textnormal{-supp}(\varphi_n(L))$. Taking the limit for $n \to +\infty$ and applying Lemma \ref{lem:gamSupLimInf}, we then obtain:
\begin{equation} \label{contra}
\gamma\textnormal{-supp}(L) \cap B(z,\eta) = \emptyset\, .
\end{equation}
Since, by hypothesis, $z \in \gamma\textnormal{-supp}(L)$, we obtain the desired contradiction. 
\end{proof}
\indent In the following, we respectively indicate with 
$$\widehat{\mathfrak{L}_c}(T^*N) \quad \text{and} \quad \widehat{\mathscr{L}_c}(T^*N)$$ 
the elements in $\widehat{\mathfrak{L}_0}(T^*N)$ with compact $\gamma$-support and the elements in $\widehat{\mathscr{L}_0}(T^*N)$ with compact $c$-support. Given $L \in \widehat{\mathfrak{L}_0}(T^*N)$, we finally want to show that $\gamma\textnormal{-supp}(L) \ne \emptyset$. This fact is a straightforward consequence of the next proposition. 
\begin{prp} \label{no empty}
If $L_1 \in \widehat{\mathfrak{L}_c}(T^*N)$ and $L_2 \in \widehat{\mathfrak{L}_0}(T^*N)$, then: \\
$(i)$ $\gamma\textnormal{-supp}(L_1) \cap \gamma\textnormal{-supp}(L_2) \ne \emptyset$. \\
$(ii)$ $\gamma\textnormal{-supp}(L_1) \cap T^*_qN \ne \emptyset$ for every $q \in N$.
\end{prp}
\noindent The proof of Proposition \ref{no empty} will make use of the following lemma, we refer to \cite[Lemma 6.12]{vit22} for a detailed proof. 
\begin{lem}\label{lem:hamTraCom}
Let $L \in \widehat{\mathfrak{L}_c}(T^*N)$ underlying $\tilde{L} \in \widehat{\mathscr{L}_c}(T^*N)$. Let $\varphi \in \mathrm{DHam}(T^*N)$ be generated by a Hamiltonian $H \equiv a$ on a neighborhood of $\gamma\textnormal{-supp}(L)$. Then $\varphi(\tilde{L}) = T_a\tilde{L}$ (and $\varphi(L) = L$).
\end{lem}
\noindent \textit{Proof of Proposition \ref{no empty}}.  $(i)$ Assume by contradiction that the two $\gamma$-supports are disjoint. Since $\gamma\textnormal{-supp}(L_1)$ is compact and $\gamma\textnormal{-supp}(L_2)$ is closed, there exists $\varphi \in \mathrm{DHam}(T^*M)$ generated by a Hamiltonian $H$ such that $H \equiv a \ne 0$ on a neighborhood of $\gamma\textnormal{-supp}(L_1)$ and $H \equiv 0$ on a neighborhood of $\gamma\textnormal{-supp}(L_2)$. Let $\tilde{L}_1 \in \widehat{\mathscr{L}_c}(T^*N)$ and $\tilde{L}_2 \in \widehat{\mathscr{L}_0}(T^*N)$ corresponding to $L_1$ and $L_2$ respectively. On the one hand, applying Proposition \ref{invariance phi}, we have that
$$c_+(\varphi(\tilde{L}_1),\varphi(\tilde{L}_2)) = c_+(\tilde{L}_1,\tilde{L}_2)\, .$$
On the other hand, by Lemma \ref{lem:hamTraCom} above:
$$c_+(\varphi(\tilde{L}_1),\varphi(\tilde{L}_2)) = c_+(T_a \tilde{L}_1,\tilde{L}_2) = c_+(\tilde{L}_1,\tilde{L}_2)-a \, .$$
Consequently $a = 0$, obtaining the desired contradiction. \\
$(ii)$ We again argue by contradiction. If $(ii)$ does not holds, we can find a small ball $B(q,\eta)\subset N$ such that 
$T^*(B(q,\eta)) \cap \gamma\textnormal{-supp}(L_1) = \emptyset.$ Take now a smooth function $f$ such that all critical points of $f$ are in $B(q,\eta)$: it is easy to construct such an $f$ up to changing the Riemannian metric on $N$. Then for any bounded set $W$ contained in the complement of $T^*
(B(q,\eta))$, we have that $\Gamma_{tf} \cap W = \emptyset$ for $t$ large enough, so that $\Gamma_{t f} \cap L_1 = \emptyset$, which contradicts $(i)$.\hfill$\Box$
\begin{es} \textnormal{Let $(f_n)_{n \ge 1}$ be a sequence in $C^{\infty}(N,\mathbb{R})$. If $\lim_{n \to \infty} f_n = f \in C^0(N,\mathbb{R})$ in the uniform norm, then $(\Gamma_{f_n})_{n \ge 1}$ is a Cauchy sequence in $\mathscr{L}_0(T^*N)$, hence defines an element $\Gamma_f \in \widehat{\mathscr{L}_0}(T^*N)$. In such a case, by \cite[Proposition 9]{AGHIV23}, we have that 
$$c\textnormal{-supp}(\Gamma_f) = \partial f\, ,$$
where $\partial f$ indicates the subdifferential for a continuous function introduced by N. Vichery in \cite[Definition 3.4]{vic13}. As a consequence of Proposition \ref{no empty}, notice that 
$$\partial f(q) = \partial f \cap T^*_q N \ne \emptyset \qquad \forall q \in N\, .$$
We also remind that if $\partial_Cf(q)$ is the Clarke differential of $f$ at $q$, we have 
$$\partial f(q) \subset \partial_C f(q) \qquad \forall q \in N\, ,$$
and the inclusion can be strict, see \cite[Theorem 3.14 and Example 3.16]{vic13} respectively. 
}\end{es}
As a consequence of $(ii)$ in Proposition \ref{no empty}, if $L \in \widehat{\mathfrak{L}_c}(T^*N)$ then $\gamma\textnormal{-supp}(L)$ may be big. In particular, that result implies that the Hausdorff dimension is greater or equal than $n$. Therefore a few natural questions arise, in particular: \\
~\newline
(A) How big can $\gamma\textnormal{-supp}(L)$ be? \\
(B) What can we say about $L$ if the Hausdorff dimension of the $\gamma\textnormal{-supp}(L)$ is $n$? \\
~\newline
Regarding question (A), we remind the reader of the existence of the so-called Peano Lagrangians $L \in \widehat{\mathfrak{L}_c}(T^*N)$ such that (in a chart)
$$\gamma\textnormal{-supp}(L) = \mathcal{O}_N \cup [0,1]^n \times [0,1]^r, \qquad 0\le r \le n\, .$$
See \cite[Theorem 7.12]{vit22}.\\
A partial answer to question (B) is given by the next result, see \cite{AGIV24}.
\begin{thm}[\cite{AGIV24}]
Let $L_{\infty} \in \widehat{\mathfrak{L}_c}(T^*N)$. We assume that $L = \gamma\textnormal{-supp}(L_{\infty})$ is a compact exact Lagrangian submanifold of $T^*N$. Then $L_{\infty} = L$ in $\mathfrak{L}_c(T^*N)$.
\end{thm}
\noindent Finally, in order to understand the $\gamma$-support of $L \in \widehat{\mathfrak{L}_c}(T^*N)$, the next example is particularly useful and addresses another (open) question: \\
~\newline
(C) If $V = \gamma\textnormal{-supp}(L)$, how many $L' \in \widehat{\mathfrak{L}_0}(T^*N)$ have the same $\gamma$-support? 
We conjecture that the number is finite if $\dim_H(V)\leq n$ (it is in fact zero if $\dim_H(V)<n$, see \cite[Prop 9.9]{vit22}). 
\begin{es} {\textnormal{Let $V$, $(L_n)_{n \ge 1}$ and $(\bar{L}_n)_{n \ge 1}$ as in Figure \ref{fig:limLag}. If we denote $L_{\infty}$ and $\bar{L}_{\infty}$ the $\gamma$-limit of $(L_n)_{n \ge 1}$ and $(\bar{L}_n)_{n \ge 1}$ respectively, it clearly holds that
$$\gamma\textnormal{-supp}(L_{\infty}) = V = \gamma\textnormal{-supp}(\bar{L}_{\infty})\, ,$$
where $\gamma(L_n,\bar{L}_n)$ is uniformly (in $n$) bounded away from $0$ (by the area of the ``loop'').}} 
\end{es}
\begin{figure}[H]
  \centering
    %% Creator: Inkscape 1.2.2 (b0a8486541, 2022-12-01), www.inkscape.org
%% Accompanies image file 'limitLagrangians.pdf' (pdf, eps, ps)
\begingroup%
  \makeatletter%
  \providecommand\color[2][]{%
    \errmessage{(Inkscape) Color is used for the text in Inkscape, but the package 'color.sty' is not loaded}%
    \renewcommand\color[2][]{}%
  }%
  \providecommand\transparent[1]{%
    \errmessage{(Inkscape) Transparency is used (non-zero) for the text in Inkscape, but the package 'transparent.sty' is not loaded}%
    \renewcommand\transparent[1]{}%
  }%
  \providecommand\rotatebox[2]{#2}%
  \newcommand*\fsize{\dimexpr\f@size pt\relax}%
  \newcommand*\lineheight[1]{\fontsize{\fsize}{#1\fsize}\selectfont}%
  \ifx\svgwidth\undefined%
    \setlength{\unitlength}{155.905501bp}%
    \ifx\svgscale\undefined%
      \relax%
    \else%
      \setlength{\unitlength}{\unitlength * \real{\svgscale}}%
    \fi%
  \else%
    \setlength{\unitlength}{\svgwidth}%
  \fi%
  \global\let\svgwidth\undefined%
  \global\let\svgscale\undefined%
  \makeatother%
  \begin{picture}(1,1.18181826)%
    \lineheight{1}%
    \setlength\tabcolsep{0pt}%
    \put(0,0){\includegraphics[width=\unitlength,page=1]{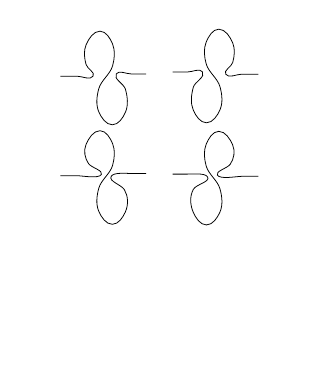}}%
    \put(0.0668262,1.0088135){\color[rgb]{0,0,0}\makebox(0,0)[lt]{\lineheight{1.25}\smash{\begin{tabular}[t]{l}$L_n$\end{tabular}}}}%
    \put(0.05700369,0.68220419){\color[rgb]{0,0,0}\makebox(0,0)[lt]{\lineheight{1.25}\smash{\begin{tabular}[t]{l}$L_m$\end{tabular}}}}%
    \put(0.76290997,1.00910925){\color[rgb]{0,0,0}\makebox(0,0)[lt]{\lineheight{1.25}\smash{\begin{tabular}[t]{l}$\bar{L}_n$\end{tabular}}}}%
    \put(0.77081338,0.69766892){\color[rgb]{0,0,0}\makebox(0,0)[lt]{\lineheight{1.25}\smash{\begin{tabular}[t]{l}$\bar{L}_m$\end{tabular}}}}%
    \put(0.12718651,0.04356787){\color[rgb]{0,0,0}\makebox(0,0)[lt]{\lineheight{1.25}\smash{\begin{tabular}[t]{l}$\gamma\textnormal{-supp}(L_{\infty})=V = \gamma\textnormal{-supp}(\bar{L}_{\infty})$\end{tabular}}}}%
    \put(0,0){\includegraphics[width=\unitlength,page=2]{limitLagrangians.pdf}}%
    \put(0.60718783,0.2174612){\color[rgb]{0,0,0}\makebox(0,0)[lt]{\lineheight{1.25}\smash{\begin{tabular}[t]{l}$V$\end{tabular}}}}%
  \end{picture}%
\endgroup%

  \caption{Convergence of two sequences of Lagrangians to different elements in the completion with same $\gamma$-support.}
  \label{fig:limLag}
\end{figure}

\section{Higher dimensional Birkhoff attractors}\label{sezione 5}
\begin{center}

\begin{minipage}{10 cm}
\textit{Contents of Section \ref{sezione 5}}. The present section is based on the paper \cite{arnHumVit24} by M.-C. Arnaud, V. Humilière and C. Viterbo. We start by briefly recalling the definition of Birkhoff attractor for a dissipative map of the cylinder. We subsequently apply the theory explained in Section \ref{sezione 4} to extend to higher dimensions the notion of Birkhoff attractor for a conformally exact symplectic diffeomorphism. After giving its definition, we prove that this new notion coincides with the classical Birkhoff attractor (on $T^*\mathbb{S}$) and we discuss some recent questions about the study of this new dynamical player.
\end{minipage}
\end{center}
\subsection{Classical Birkhoff attractors}
Let $\psi$ be a conformally exact symplectic (CES) diffeomorphism on $T^*\mathbb{S} = \mathbb{S} \times \mathbb{R}$ with conformal ratio $a \in (0,1)$:
$$\psi^*pdq - a pdq = df\, .$$
Moreover, suppose that $\psi$ is a diffeomorphism of $\mathbb{S} \times [-1,1]$ into its image
\begin{equation} \label{diss diffeo}
\psi(\mathbb{S} \times [-1,1]) \subset \mathbb{S} \times ]-1,1[\, ,
\end{equation}
which is homotopic to the identity. In what follows, we will briefly call $\psi$ with these properties a dissipative map. As a consequence of (\ref{diss diffeo})we may set: 
\begin{defn} \label{Conley}
The attractor (\`a la Conley) of $\psi$ is the set
\[
B_0 = B_0(\psi):= \bigcap_{n\in\mathbb{N}}\psi^n(\mathbb{S} \times [-1,1])\, .
\]
\end{defn} 
\noindent Observe that $B_0$ is a compact, non-empty, connected, $\psi$-invariant set. Moreover, it separates  $\mathbb{S} \times [-1,1]$, i.e., its complementary set is the union of two open, disjoint, invariant sets $U_+$ and $U_-$. However, the attractor $B_0$ does not have to be ``minimal'' in the sense that it could be a priori further decomposed into smaller invariant pieces. This fundamental observation --which will be clarified in Proposition \ref{Birk-attr-minimal}--  justifies the following definition, introduced by G.D. Birkhoff in \cite{bir32}.
\begin{defn}
 Let $\psi: \mathbb{S} \times [-1 , 1] \to \mathbb{S} \times ]-1,1[$ be a dissipative map and let $B_0$ be its attractor, so that $\mathbb{S} \times [-1,1] \setminus B_0 = U_-\sqcup U_+$. The Birkhoff attractor of $\psi$ is then
 \[
 B = B(\psi):= \mathrm{cl}(U_-)\cap \mathrm{cl}(U_+)\, .
 \]
\end{defn}
One should be careful about the abuse of language: $B$ is not necessarily an attractor because it may not be true that $\lim_{n \to +\infty}(\psi^n(q,p),B)=0$ for all $(q,p) \in T^*\mathbb S$. \\
\noindent It is clear that since $U_\pm$ are both invariant, so is $B$. Moreover $B$ is connected and separates the annulus, in the sense that $\mathbb{S} \times [-1 , 1] \setminus B$ is the disjoint union of two open invariant sets containing $\sphere\times \{\pm 1\}$ respectively. Denote by $\mathcal{X}$ the set of compact, connected, invariant subsets of $\mathbb{S} \times [-1 , 1]$ that separate the annulus. For the proof of the next result and for further details, we refer to \cite{lec88}.
\begin{prp}
\label{Birk-attr-minimal}
The Birkhoff attractor $B \in \mathcal{X}$ and it is the smallest element of $\mathcal{X}$ with respect to the inclusion.
\end{prp}
\begin{figure}[ht]
    \centering
    \includegraphics[scale=0.7]{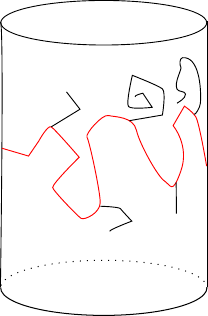}
    \caption{The red set corresponds to the Birkhoff attractor, while the union of the black and red sets gives the attractor à la Conley.}
    \label{Birk-no-attr}
\end{figure}
\begin{es}\label{es-semplice}
{\textnormal{The easiest example of dissipative map is 
$$\mathbb{S}\times [-1,1] \ni (q,p) \to (q,ap) \in \mathbb{S} \times ]-1,1[\, ,$$
where $a\in(0,1)$. Clearly, in such a case, the attractor coincides with the Birkhoff attractor, i.e. $B_0 = B = \mathbb{S}\times\{0\}$. There are in fact examples where the Birkhoff attractor is strictly contained in the attractor, as shown in Figure \ref{Birk-no-attr}.}}
\end{es}
\begin{es} {\textnormal{In the case of the damped pendulum: 
$$\ddot{\theta} = \sin \theta - \lambda \dot{\theta}\, ,$$
$\lambda \in \mathbb{R}_{>0}$, the Birkhoff attractor, which coincides with the attractor, is given by the closure of the unstable manifold of the hyperbolic fixed point $(\pi,0)$, that is
$$B_0 = B = \textnormal{cl}(W^u(\pi,0)) = W^u(\pi,0) \cup (0,0).$$ 
We refer to Figure \ref{fig:attractorPendulum}.}}
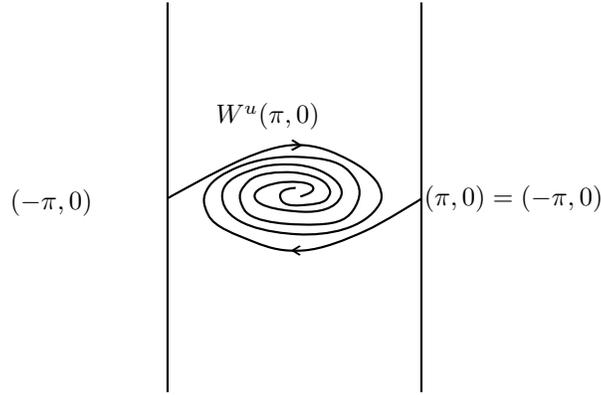
\begin{figure}[ht]
  \centering
    %% Creator: Inkscape 1.2.2 (b0a8486541, 2022-12-01), www.inkscape.org
%% Accompanies image file 'attractorPendulum.pdf' (pdf, eps, ps)
\begingroup%
  \makeatletter%
  \providecommand\color[2][]{%
    \errmessage{(Inkscape) Color is used for the text in Inkscape, but the package 'color.sty' is not loaded}%
    \renewcommand\color[2][]{}%
  }%
  \providecommand\transparent[1]{%
    \errmessage{(Inkscape) Transparency is used (non-zero) for the text in Inkscape, but the package 'transparent.sty' is not loaded}%
    \renewcommand\transparent[1]{}%
  }%
  \providecommand\rotatebox[2]{#2}%
  \newcommand*\fsize{\dimexpr\f@size pt\relax}%
  \newcommand*\lineheight[1]{\fontsize{\fsize}{#1\fsize}\selectfont}%
  \ifx\svgwidth\undefined%
    \setlength{\unitlength}{184.2519685bp}%
    \ifx\svgscale\undefined%
      \relax%
    \else%
      \setlength{\unitlength}{\unitlength * \real{\svgscale}}%
    \fi%
  \else%
    \setlength{\unitlength}{\svgwidth}%
  \fi%
  \global\let\svgwidth\undefined%
  \global\let\svgscale\undefined%
  \makeatother%
  \begin{picture}(1,0.92307692)%
    \lineheight{1}%
    \setlength\tabcolsep{0pt}%
    \put(0,0){\includegraphics[width=\unitlength,page=1]{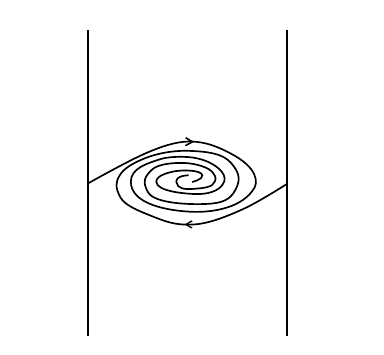}}%
    \put(-0.09232773,0.42331337){\color[rgb]{0,0,0}\makebox(0,0)[lt]{\lineheight{1.25}\smash{\begin{tabular}[t]{l}$(-\pi, 0)$\end{tabular}}}}%
    \put(0.75380052,0.43045041){\color[rgb]{0,0,0}\makebox(0,0)[lt]{\lineheight{1.25}\smash{\begin{tabular}[t]{l}$(\pi, 0)=(-\pi, 0)$\end{tabular}}}}%
    \put(0.32785607,0.59936517){\color[rgb]{0,0,0}\makebox(0,0)[lt]{\lineheight{1.25}\smash{\begin{tabular}[t]{l}$W^u(\pi, 0)$\end{tabular}}}}%
  \end{picture}%
\endgroup%

  \caption{The Birkhoff attractor for the damped pendulum.}
    \label{fig:attractorPendulum}
\end{figure}
\end{es}
\noindent The study of the dynamical as well as the topological complexity of $B$ passes through the notion of upper and lower rotation number $\rho^{\pm}$ of the Birkhoff attractor. To give and idea of $\rho^{\pm}$, let $V_{\pm}(q,p) := \set{ (q,y)\in T^*\mathbb{S}: \pm(y-p) \ge 0}$. Then $\rho^\pm$ are the rotation numbers of the points in the ``upper and lower graphs'' of $B$, that is 
$$B_{\pm} := \set{ (q,p) \in B: V_{\pm}(q,p) \setminus \{(q,p)\} \subseteq U_{\pm}}\, .$$
We stress that the definition of the Birkhoff attractor does not need any twist property. However, in order to define $\rho^\pm$, $\psi$ must be a dissipative twist map. We refer to \cite{lec88} for further details on $\rho^{\pm}$. \\
We conclude this brief discussion on the classical Birkhoff attractor by a fundamental result due to M. Charpentier \cite{cha34}, providing a sufficient condition on the upper and lower rotation number $\rho^{\pm}$ for the existence of a ``complex'' Birkhoff attractor.
\begin{figure}
\centering
\includegraphics[height=4cm, width=5cm,
keepaspectratio]{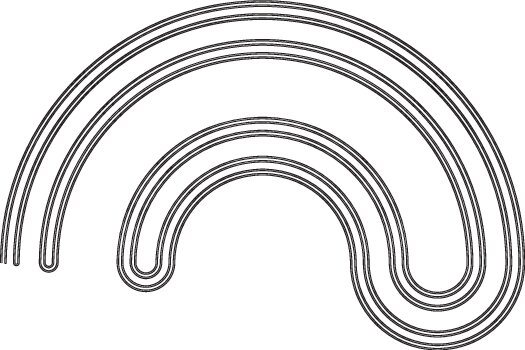}
\caption{An example of an indecomposable continuum (L Rempe-Gillen, CC BY-SA 3.0 https://creativecommons.org/licenses/by-sa/3.0, via Wikimedia Commons).}
\label{ind}
\end{figure}
\begin{thm} \label{charp}
Let $\psi: \mathbb{S} \times [-1 , 1] \to \mathbb{S} \times ]-1,1[$ be a dissipative twist map. If $\rho^+-\rho^->0$, then the corresponding Birkhoff attractor is an indecomposable continuum, i.e. it cannot be written as the union of two compact, connected, non-trivial sets, see Figure \ref{ind}.
\end{thm}
\indent In order to proceed with the definition of the generalized Birkhoff attractor, we need to recall an immediate consequence of Proposition \ref{conformal invariance phi} which will play an important role in the sequel.
\begin{prp} Let $\psi \in \mathrm{Diff}(T^*N)$ be a conformally exact symplectic (CES) diffeomorphism with conformal ratio $a \in (0,1)$. Then, for every $(L_1,L_2) \in \widehat{\mathfrak{L}_0}(T^*N) \times \widehat{\mathfrak{L}_0}(T^*N)$ it holds that
$$\gamma(\psi(L_1),\psi(L_2)) = a \gamma(L_1,L_2)\, .$$
\end{prp}
\noindent As a consequence of this proposition, and by Banach-Caccioppoli’s fixed point theorem,
$$\psi: \widehat{\mathfrak{L}_0}(T^*N) \to \widehat{\mathfrak{L}_0}(T^*N)$$
has a unique fixed point $L_{\infty} = L_{\infty}(\psi)$. 
\begin{defn}
Let $\psi \in \mathrm{Diff}(T^*N)$ be a conformally exact symplectic (CES) diffeomorphism with conformal ratio $a \in (0,1)$. The subset
$$B_{\infty} = B_{\infty}(\psi) = \gamma\textnormal{-supp}(L_{\infty})$$
is the generalized Birkhoff attractor of $\psi$. 
\end{defn}
\noindent The main properties of $B_{\infty}$ are resumed in the next theorem.
\begin{thm}
Let $\psi \in \mathrm{Diff}(T^*N)$ be a conformally exact symplectic (CES) diffeomorphism with conformal ratio $a \in (0,1)$. The Birkhoff attractor $B_{\infty}$ of $\psi$ is an invariant, closed and $\gamma$-coisotropic subset of $T^*N$. 
\end{thm}
\begin{proof}
Closeness and invariance of $B_{\infty}$ are points $(ii)$ and $(iii)$ of Remark \ref{3remarks} respectively. The fact that $B_{\infty}$ is $\gamma$-coisotropic in $T^*N$ is exactly Proposition \ref{gamma coisotropic}. 
\end{proof}
\begin{rem} \textnormal{Let $D_rT^*N := \Set{(q, p)\in T^*N: \vert p \vert_{g^*} \leq r}$, where $g^*$ is the dual metric induced by a Riemannian metric $g$ on $N$. We additionally observe that $B_{\infty} \in \widehat{\mathfrak{L}_c}(T^*N)$ provided we assume $\psi(D_rT^*N) \subset D_rT^*N$.}
\end{rem}
\begin{rem}\textnormal{
Let now return to the case of the cylinder $\mathbb{S} \times [-1,1]$. By $(ii)$ of Proposition \ref{no empty}, $B_{\infty}(\psi) = \gamma\textnormal{-supp}(L_{\infty}(\psi))$ intersects all vertical fibers of $\mathbb{S} \times [-1,1]$ so that $B_{\infty}(\psi)$ is an annular set that is separating. Since $B_{\infty}(\psi)$ is also compact and invariant, we have that the $B_{\infty}(\psi) \subseteq B_0(\psi)$, see Definition \ref{Conley} for $B_0(\psi)$. It is worth noting that $B_{\infty}(\psi)$ can be strictly contained in $B_0(\psi)$, because $B_0(\psi)$ can be non $\gamma$-coisotropic at certain points (e.g. at the end of the ``hair'' of Figure \ref{Birk-no-attr}) for the same reason that $[0,1] \times \{0\}$ is not $\gamma$-coisotropic at $(0,0)$ and $(1,0)$. We refer to the Examples of Section \ref{prp:coiImpGamCoi}}.
\end{rem}}}
Here below we prove that the generalized Birkhoff attractor $B_{\infty}$ coincides with the classical Birkhoff attractor $B$ in $T^*\mathbb{S}$.
\begin{thm} Let $\psi \in \mathrm{Diff}(T^*\mathbb{S})$ be a dissipative map with conformal ratio $a \in (0,1)$. Then $B_\infty(\psi) = B(\psi)$.
\end{thm}
\begin{proof}
Let us start by proving the inclusion $B_{\infty}(\psi) \subseteq B(\psi)$. Since $B(\psi)$ is defined by $B(\psi) = \textnormal{cl}(U_-) \cap \textnormal{cl}(U_+)$, we need to show that $B_{\infty}(\psi)$ is contained in both $\textnormal{cl}(U_\pm)$. \\
For every $\alpha \in [-1,1]$, we indicate by $\mathcal{L}_\alpha(\mathbb{S} \times [-1,1])$ the set of simple, homotopically non trivial curves with Liouville class $\alpha \in H^1(\mathbb{S};\mathbb{R}) \cong \mathbb{R}$. Remark that by definition these Lagrangians are not exact unless $\alpha=0$. Moreover, let denote by $\tau_{\alpha}:(q,p) \mapsto (q,p+\alpha)$ the $\alpha$-translation on the fiber. Then $\psi$ sends $\mathcal{L}_\alpha(\mathbb{S} \times [-1,1])$ to $\mathcal{L}_{a\alpha}(\mathbb{S} \times [-1,1])$ and $\tau_{\alpha}$ sends $\mathcal{L}_0(\mathbb{S} \times [-1,1])$ to $\mathcal{L}_{\alpha}(\mathbb{S} \times [-1,1])$. 
Given an element $\Lambda_0 \in \mathcal{L}_{\alpha}(\mathbb{S} \times [-1,1])$, we consider the sequence:
$$\Lambda_n := \tau_{-a^n\alpha} \circ \psi^n(\Lambda_0)\, , \qquad n \ge 1\, .$$
Every $\Lambda_n $ belongs to $\mathcal{L}_0(\mathbb{S} \times [-1,1])$ and hence to $\mathfrak{L}_0(\mathbb{S} \times [-1,1])$ because the Nearby Lagrangian Conjecture is true for $T^*\mathbb{S}$ --see the Exercise at the end of this proof. Moreover, if we define 
\begin{equation} \label{iterazione}
f_n := \tau_{-a^n\alpha} \circ \psi \circ \tau_{a^{n-1} \alpha}, \qquad n \ge 1\ ,
\end{equation}
it is easy to show that
$$
\Lambda_n = f_n(\Lambda_{n-1}), \qquad n \ge 1\, .
$$
Consider the maps $\psi$ and $f_n$ acting from the complete space $(\widehat{\mathfrak{L}_0}(\mathbb{S} \times [-1,1]),\gamma)$ into itself. They are all contractions of ratio $a$ and moreover $f_n \stackrel{\gamma}{\to} \psi$. Consequently, we can apply the fixed point result given by Proposition 3.8 in \cite{arnHumVit24}, guaranteeing that for every $L \in \widehat{\mathfrak{L}_0}(\mathbb{S} \times [-1,1])$, the sequence $f_n \circ \ldots \circ f_2 \circ f_1 (L)$ $\gamma$-converges to the (unique) fixed point of $\psi$, which is $L_{\infty}(\psi)$. \\
Take now $\Lambda_0 = \mathbb{S} \times \{\alpha\}$ for $\alpha$ so close to $1$ that $\mathbb{S} \times \{\alpha\} \subset U_+$ and apply the above convergence result to $\Lambda_1 = f_1(\Lambda_0) \in \mathfrak{L}_0(\mathbb{S} \times [-1,1])$ and to the sequence $(f_n)_{n \ge 2}$ defined as in (\ref{iterazione}). We then obtain that the sequence $(\Lambda_n)_{n \ge 1}$ in $\mathfrak{L}_0(\mathbb{S} \times [-1,1])$, defined iteratively by $\Lambda_n = f_n(\Lambda_{n-1})$, $\gamma$-converges to $L_{\infty}(\psi) \in \widehat{\mathfrak{L}_0}(\mathbb{S} \times [-1,1])$. By Lemma \ref{lem:gamSupLimInf} we find the first inclusion in the following expression:
\begin{eqnarray*}
 & B_{\infty}(\psi) = \gamma\textnormal{-supp}(L_{\infty}(\psi)) \subseteq
 \liminf \gamma\textnormal{-supp}(\Lambda_n) = \lim \Lambda_n \\ & = \liminf \psi^n(\mathbb{S} \times \{\alpha\}) \subseteq \liminf \textnormal{cl}(U_{n,+}) = \textnormal{cl}(U_+)\, ,     
\end{eqnarray*}
where $U_{n,+}$ indicates the connected component of $\mathbb{S} \times [-1,1] \setminus (\psi^n(\mathbb{S} \times [-1,1]))$ containing $\mathbb{S} \times \{1\}$, so that $U_+ = \bigcup_{n \ge 1} U_{n,+}$. \\
We remark that the equality $ \liminf \gamma\textnormal{-supp}(\Lambda_n) = \lim \Lambda_n$ is true since $(\Lambda_n)_{n \ge 1}$ is a sequence in $\mathfrak{L}_0(\mathbb{S} \times [-1,1])$. On the other hand, the equality $\lim \Lambda_n = \liminf \psi^n(\mathbb{S} \times \{\alpha\})$ holds since $\Lambda_n = \tau_{-a^n \alpha} \circ \psi^n(\mathbb{S} \times \{ a \})$ and $\tau_{-a^n \alpha} \circ \psi^n \stackrel{C^0}{\rightarrow} \psi^n$. 
\newline
Similarly, taking $\Lambda_0 = \mathbb{S} \times \{-\alpha\}$ for $\alpha$ close enough to $1$ in order to have
$\mathbb{S} \times \{-\alpha\} \subset U_-$, we have that
$$B_{\infty}(\psi) = \gamma\textnormal{-supp}(L_{\infty}(\psi)) \subseteq \textnormal{cl}(U_-)\, .$$
We then obtain the desired inclusion:
$$B_{\infty}(\psi) \subseteq \textnormal{cl}(U_+) \cap \textnormal{cl}(U_-) = B(\psi)\, .$$
\indent We are now left with the task of proving the reverse inclusion, i.e. that $B(\psi) \subseteq B_{\infty}(\psi)$. By $(ii)$ of Proposition \ref{no empty}, $B_{\infty}(\psi) = \gamma\textnormal{-supp}(L_{\infty}(\psi))$ intersects all vertical fibers of $\mathbb{S} \times [-1,1]$, hence is separating $\mathbb{S} \times [-1,1]$:
$$(\mathbb{S} \times [-1,1]) \setminus B_{\infty}(\psi) = W_+ \sqcup W_-\, .$$
We notice that there cannot be another bounded connected component otherwise it would be invariant and hence of zero measure, since $\psi$ is dissipative with conformal ratio $a \in (0,1)$. We already know that $B_{\infty}(\psi) \subseteq B(\psi)$. Then 
$$\underbrace{(\mathbb{S} \times [-1,1]) \setminus B(\psi)}_{=U^1_+ \sqcup U^1_-} \subseteq \underbrace{(\mathbb{S} \times [-1,1]) \setminus B_{\infty}(\psi)}_{= W_+ \sqcup W_-}\, .$$
Hence $U^1_{\pm} \subseteq W_{\pm}$. In order to conclude, it is sufficient to prove that the previous inclusions are equalities. Without loss of generality, suppose on the contrary that there exists $z = (q,p) \in W_+ \setminus U^1_+$. Then there exists $\varepsilon > 0$ such that $B(z,\varepsilon) \subseteq W_+ \Rightarrow d(z,W_-) \ge \varepsilon$. Since $U_-^1 \subseteq W_-$, we have also $d(z,U^1_{-}) \ge \varepsilon$. But if $z \notin U^+_1$ then we must have $z \in \textnormal{cl}(U_-^1)$ and therefore $d(z,U_-^1) = 0$, which is the desired contradiction. So $U^1_{\pm} = W_{\pm}$ and $B(\psi) \subseteq B_{\infty}(\psi)$. \\
\indent Another way to prove the above inclusion is using the fact that $B_{\infty}(\psi)$ is connected in the case of a cotangent bundle, see \cite{AGIV24}. Consequently, $B_{\infty}(\psi) \in \mathcal{X}$, where $\mathcal{X}$ is the set of compact, connected, invariant subsets of $\mathbb{S} \times [-1 , 1]$ that separate the annulus. As a consequence of Proposition \ref{Birk-attr-minimal}, the desired inclusion immediately follows. 
\end{proof}
\begin{exo}\label{ex:nlc}
\textnormal{Show the Nearby Lagrangian Conjecture on $T^*\sphere$: an exact Lagrangian in $(T^*\mathbb{S}, -d(pdq))$ is Hamiltonian isotopic to $\mathcal{O}_\sphere$. Prove at first the easy case where our exact Lagrangian is the graph of the differential of a function $f:\sphere\rightarrow \R$. Then, consider a general exact Lagrangian $L_1$: it is an essential curve that bounds zero area with the zero section. Consider any isotopy $L:[0, 1]\times\sphere\rightarrow T^*\sphere$ via embedded, essential curves with image $\mathcal{O}_\sphere$ and $L_1$ for $t = 0$ and $t = 1$ respectively: it is a Lagrangian isotopy. Show that this isotopy may be assumed to be in exact Lagrangians. Remark that if $t_0$ and $t_1$ are close enough (without loss of generality $t_0<t_1$) then $L(t_1, \cdot)$ is contained in a Weinstein neighborhood of $L(t_0, \cdot)$. Prove that it is the graph of the differential of a function in this neighborhood and apply the first part. Now, iterate this construction and prove the Nearby Lagrangian Conjecture on $T^*\sphere$.}
\end{exo}
As explained in Proposition \ref{Birk-attr-minimal} and Theorem \ref{charp}, the topology of the classical Birkhoff attractor is rather well understood. The following result partially answers the question on general topological properties of $B_{\infty}(\psi)$. In particular, it generalizes a property of the classical Birkhoff attractor $B(\psi)$, which is cohomologically non-trivial, in the sense that if we restrict the generator of $H^1(\sphere)$ to any open neighborhood of $B(\psi)$, it does not vanish.
\begin{thm}
Let $\psi \in \mathrm{Diff}(T^*N)$ be a conformally exact symplectic (CES) diffeomorphism such that $B_{\infty}(\psi)$ is compact. Let $\pi: B_{\infty}(\psi) \rightarrow N$ be the canonical projection. Then 
$$\pi^*:H^*(N)\rightarrow H^*(B_{\infty}(\psi))$$ 
is injective, where \begin{equation*}
    H^*(B_{\infty}(\psi)):=\lim_{\substack{\longrightarrow \\ U\supset B_{\infty}(\psi) \\ U \ \mathrm{open}}} H^*(U)\, .
\end{equation*}
\end{thm}
\begin{proof}
We need to prove that for any open neighborhood $U$ of $B_{\infty}(\psi)$, the map $H^{\ell}(N) \to H^{\ell}(U)$ is injective for any integer $\ell \ge 0$. We may assume without loss of generality that the $\textnormal{cl}(U)$ is compact. Moreover, since this map sends the class $1 \in H
^0(N)$ to $1 \in H^0(U)$, we may assume that $\ell > 0$. \\
Let $V$ be a closed neighborhood of $B_{\infty}(\psi)$ contained in $U$. Define a Hamiltonian function $H \in C^\infty_c(T^*N; \R)$ satisfying the following properties for a small $\varepsilon>0$:
\begin{enumerate}
\item $H$ is $C^2$-small.
\item $H$ is equal to $-\varepsilon$ on $V$.
\item $H > -\varepsilon$ in $T^*N\setminus V$.
\item $H = 0$ in $T^*N\setminus U$.
\end{enumerate}
In order to prove the theorem, given $\alpha \in H^{\ell}(N)$ a non-zero cohomology class of positive degree, we need to verify that it does not vanish on $U$. Let $c(\alpha, \varphi_H)$ be the Hamiltonian spectral invariant of $\varphi_H$, defined by using Hamiltonian Floer theory, see \cite{FS07} and also \cite{Lan16}. This invariant is related to the spectral invariant of Definition \ref{serve in ultima} by the next inequality, see \cite[Proposition 2.14]{monVicZap12} and also \cite[Lemma 6.1]{arnHumVit24} for its cohomological counterpart:
    \begin{equation}\label{open closed}
        c(\alpha, \varphi_H)\leq c(\alpha, L, \varphi_H(L))\, .
    \end{equation}
Moreover, since $H$ is sufficiently $C^2$-small and autonomous, and the only periodic orbits of $\varphi_H$ are fixed throughout the isotopy and have action $-\varepsilon$ and $0$, it holds that 
\begin{equation} \label{morse spectral}
c(\alpha, \varphi_H) = c(\alpha,H)\, ,
\end{equation}
where $c(\alpha,H)$ denotes the Morse theoretic spectral invariant of $H$, we refer to \cite[Page 29]{arnHumVit24} for its detailed definition. \\
Now we implicitly pass to a brane of $L$ and use Lemma \ref{lem:hamTraCom}. Since $\varphi_H(L)=T_{-\varepsilon}L$, we have that
\begin{equation}\label{eqn:minH}
c(\alpha, L,\varphi_H(L))=c(\alpha, L, T_{-\varepsilon}L)=c(\alpha, L, L)-\varepsilon=-\varepsilon=\min H\, .
\end{equation}
By using (\ref{open closed}), (\ref{morse spectral}) and (\ref{eqn:minH}), we then obtain that
$$c(\alpha,H) \le \min H\, .$$
Recalling that $c(\alpha,H) \in \{0,\min H\}$, the unique possibility is that $c(\alpha,H) = \min H$. Since $c(1,H) = \min H$, we then obtain that $c(\alpha,H) = c(1,H)$. Applying Lusternik-Schnirelman theory --see \cite[Lemma 6.1]{arnHumVit24} and also Theorem \ref{thm:lusSch}-- we find that $\alpha$ induces a non-zero class in every neighborhood of $H^{-1}(-\varepsilon)$. In particular, $\alpha$ induces a non-zero class in $U$, which concludes the proof of the injectivity of the map $H^{\ell}(N) \to H^{\ell}(U)$.
\end{proof}
\noindent We notice that the following question remains open: Is $B_{\infty}$ always connected? This was recently answered positively in the case of a cotangent bundle in \cite{AGIV24}. It is also proved that $\gamma\textnormal{-supp}(L)$ does not have to be connected when non-compact. Finally, also questions regarding the dynamical complexity of $B_{\infty}$ are open and need to be exploited. \\
\indent To conclude this section, we remind that a natural dissipative system is the one associated to the discounted Hamilton-Jacobi equation:
$$\alpha u(x) + H(x,du(x)) = 0, \qquad \alpha > 0\, .$$
In fact, the corresponding time-1 map $\varphi^1_{H,\alpha}$ is conformally symplectic with conformal ratio $a = e^{-\alpha}$. In \cite[Theorem 1.5]{arnHumVit24} it is proved that, for a Tonelli Hamiltonian $H$ on $T^*N$, the graph of the corresponding viscosity solution $(x,du_{H,\alpha}(x))$ belongs to $B_{\infty}(\varphi^1_{H,\alpha})$. Moreover, in the Appendix A of the same paper, M. Zavidovique constructs an example of a Hamiltonian (non Tonelli) for which the conclusion of the above theorem does not hold, as well as an example of a
Tonelli Hamiltonian for which the generalized Birkhoff attractor of $\varphi^1_{H,\alpha}$ is strictly larger than the closure of $\bigcup_{t \in \mathbb{R}} \varphi^t_{H,\alpha} (\textnormal{Graph}(du_{H,\alpha}))$.

\section{Concluding remarks}
We saw in the specific example of conformally symplectic maps how the introduction of the Humilière completion of the space of exact Lagrangian submanifolds has deep and far reaching applications to the existence of invariants subsets and their properties that could only be previously evidenced in dimension $2$. This Humilière completion, together with its Hamiltonian counterpart, the completion $\widehat{\mathrm{DHam}}(T^*N)$ have many other applications, and we are convinced that many more are just waiting to be found. 

\begin{enumerate}
\item One can show via the methods we described here that the action of $\widehat{\mathrm{DHam}}(T^*N)$ is transitive on  $\mathfrak{L}(T^*N)$. In fact, the next weak version of the Nearby Lagrangian Conjecture holds:
\begin{thm}
Let $L,L'$ be two closed exact Lagrangians in $T^*N$. Then there exists $\psi \in \widehat{\mathrm{DHam}}(T^*N)$ such that $\psi(L) = L'$.  
\end{thm}
We refer to \cite[Section 7]{arnHumVit24} for all details. Remind that the Nearby Lagrangian Conjecture states that the restriction of the action to $\mathrm{DHam}(T^*N)$ is still transitive.
\item It would be interesting to describe the elements in $\widehat{\mathrm{DHam}}(T^*N)$. So far, we know that the $C^0$-closure of $\mathrm{DHam}(T^*N)$ is contained in $\widehat{\mathrm{DHam}}(T^*N)$, just like the flows of continuous Hamiltonian functions which are discontinuous at submanifolds of codimension at least 1, and Hamiltonians on the cotangent of a torus obtained via homogenization: if $H \in C^\infty(T^*\mathbb{T}^n)$, define $H_k(q, p):=H(kq, p)$ and the Hamiltonian diffeomorphisms generated by $H_k$ admit a $\gamma$-limit in $\widehat{\mathrm{DHam}}(T^*N)$. We refer to \cite{SymplHo} for an exhaustive discussion of this topic. 
\item For a Liouville manifold, the Liouville vector field defines a Liouville flow that has an attractor (see \cite{Biran-Skeleta}) which --if the manifold is Weinstein-- is a Lagrangian piecewiselinear subset. Then we have as many (generalized) Birkhoff attractors as there are equivalence classes (for the Floer-Fukaya equivalence relation) of Lagrangians, and these Birkhoff attractors are contained in the skeleton, which allows one in some cases to compute the Floer-Fukaya equivalence classes. 
\end{enumerate}
\printbibliography

@article{Biran-Skeleta,
	author = {Paul Biran},
	journal = {Geometric and Functional Analysis},
	pages = {279-326},
	title = {Lagrangian non-intersections},
	volume = {16},
	year = {2006}}

@book{Bott-Tu,
	author = {Bott, R. and Tu, L.},
	publisher = {Springer-Verlag},
	series = {GTM 82},
	title = {Differential forms in Algebraic Topology},
	year = {1982}}

@book{Spanier,
	author = {Spanier, Edwin},
	publisher = {Springer-Verlag},
	title = {Algebraic topology},
	year = {1966}}

@article{Lalonde-Sikorav,
	author = {Lalonde, Fran{\c c}ois and Sikorav, Jean-Claude},
	journal = {Commentarii Mathematici Helvetici},
	pages = {18-33},
	title = {Sous-vari{\'e}t{\'e}s lagrangiennes et lagrangiennes exactes des fibr{\'e}s cotangents},
	volume = {66},
	year = {1991}}

@misc{Viterbo1995b,
	author = {Viterbo, Claude},
	eprint = {1805.01316},
	eprinttype = {arxiv},
	primaryclass = {math.SG},
	title = {Functors and Computations in Floer cohomology, II}}

@article{Abouzaid-Courte-Guillermou-Kragh,
	author = {Mohammed Abouzaid and Sylvain Courte and St{\'e}phane Guillermou and Thomas Kragh},
	journal = {Duke Mathematical Journal},
	number = {5},
	pages = {949-1011},
	title = {Twisted generating functions and the nearby Lagrangian conjecture},
	volume = {174},
	year = {2025}}

@misc{ACLS,
	author = {Marcelo S. Atallah and Jean-Philippe Chass{\'e} and R{\'e}mi Leclercq and Egor Shelukhin},
	eprint = {2410.04158},
	eprinttype = {arXiv},
	primaryclass = {math.SG},
	title = {Weinstein exactness of nearby Lagrangians and related questions},
	year = {2024}}

@article{arnFéj,
 author = {Arnaud, Marie-Claude and Fejoz, Jacques},
 title = {Invariant submanifolds of conformal symplectic dynamics},
 fjournal = {Journal de l'{\'E}cole Polytechnique -- Math{\'e}matiques},
 journal = {J. {\'E}c. Polytech., Math.},
 issn = {2429-7100},
 volume = {11},
 pages = {159--185},
 year = {2024},
 language = {English},
 doi = {10.5802/jep.252},
 zbMATH = {7811891},
 Zbl = {1541.37062}
}

@book{mcDSal16,
 author = {McDuff, Dusa and Salamon, Dietmar},
 title = {Introduction to symplectic topology},
 edition = {3rd edition},
 fseries = {Oxford Graduate Texts in Mathematics},
 series = {Oxf. Grad. Texts Math.},
 volume = {27},
 isbn = {978-0-19-879489-9; 978-0-19-879490-5},
 year = {2016},
 publisher = {Oxford: Oxford University Press},
 language = {English},
 zbMATH = {6638013},
 Zbl = {1380.53003}
}

@article{the99,
 author = {Th{\'e}ret, David},
 title = {A complete proof of {Viterbo}'s uniqueness theorem on generating functions},
 fjournal = {Topology and its Applications},
 journal = {Topology Appl.},
 issn = {0166-8641},
 volume = {96},
 number = {3},
 pages = {249--266},
 year = {1999},
 language = {English},
 doi = {10.1016/S0166-8641(98)00049-2},
 zbMATH = {1340387},
 Zbl = {0952.53037}
}

@article{vit92,
 author = {Viterbo, Claude},
 title = {Symplectic topology as the geometry of generating functions},
 fjournal = {Mathematische Annalen},
 journal = {Math. Ann.},
 issn = {0025-5831},
 volume = {292},
 number = {4},
 pages = {685--710},
 year = {1992},
 language = {English},
 doi = {10.1007/BF01444643},
 url = {https://eudml.org/doc/164936},
 zbMATH = {4215961},
 Zbl = {0735.58019}
}

@book{mil63,
 author = {Milnor, John W.},
 title = {Morse theory. {Based} on lecture notes by {M}. {Spivak} and {R}. {Wells}},
 fseries = {Annals of Mathematics Studies},
 series = {Ann. Math. Stud.},
 volume = {51},
 year = {1963},
 publisher = {Princeton University Press, Princeton, NJ},
 language = {English},
 doi = {10.1515/9781400881802},
 zbMATH = {3176330},
 Zbl = {0108.10401}
}

@book{spa95,
 author = {Spanier, Edwin H.},
 title = {Algebraic topology},
 isbn = {0-387-94426-5},
 year = {1995},
 publisher = {Berlin: Springer-Verlag},
 language = {English},
 zbMATH = {715014},
 Zbl = {0810.55001}
}

@article{milOh97,
 author = {Milinkovi{\'c}, Darko and Oh, Yong-Geun},
 title = {Floer homology as the stable {Morse} homology},
 fjournal = {Journal of the Korean Mathematical Society},
 journal = {J. Korean Math. Soc.},
 issn = {0304-9914},
 volume = {34},
 number = {4},
 pages = {1065--1087},
 year = {1997},
 language = {English},
 zbMATH = {1139387},
 Zbl = {0921.58019}
}

@book{hum08,
 author = {Humili{\`e}re, Vincent},
 title = {Continuit{\'e} en topologie symplectique},
 year = {2008},
 publisher = {Paris: {\'E}cole Polytechnique (Diss.)},
 language = {French},
 zbMATH = {7024031},
 Zbl = {1406.37002}
}

@misc{vic13,
 author = {Vichery, Nicolas},
 title = {Homological differential calculus},
 year = {2013},
 howpublished = {Preprint, {arXiv}:1310.4845 [math.{AT}] (2013)},
 url = {https://arxiv.org/abs/1310.4845},
 arXiv = {arXiv:1310.4845}
}

@article{ban78,
 author = {Banyaga, Augustin},
 title = {Sur la structure du groupe des diff{\'e}omorphismes qui preservent une forme symplectique},
 fjournal = {Commentarii Mathematici Helvetici},
 journal = {Comment. Math. Helv.},
 issn = {0010-2571},
 volume = {53},
 pages = {174--227},
 year = {1978},
 language = {French},
 doi = {10.1007/BF02566074},
 url = {https://eudml.org/doc/139732},
 zbMATH = {3610489},
 Zbl = {0393.58007}
}

@misc{vit22,
 author = {Viterbo, Claude},
 title = {On the supports in the {Humili{\`e}re} completion and ${{\gamma}}$-coisotropic sets},
 year = {2022},
 howpublished = {Preprint, {arXiv}:2204.04133 [math.{SG}] (2022)},
 url = {https://arxiv.org/abs/2204.04133},
 arXiv = {arXiv:2204.04133}
}

@misc{arnHumVit24,
 author = {Arnaud, Marie-Claude and Humili{\`e}re, Vincent and Viterbo, Claude},
 title = {Higher {Dimensional} {Birkhoff} attractors (with an appendix by {Maxime} {Zavidovique})},
 year = {2024},
 howpublished = {Preprint, {arXiv}:2404.00804 [math.{SG}] (2024)},
 url = {https://arxiv.org/abs/2404.00804},
 arXiv = {arXiv:2404.00804}
}

@article{cha34,
 author = {Charpentier, M.},
 title = {Sur quelques propri{\'e}t{\'e}s des courbes de {M}. {Birkhoff}},
 fjournal = {Bulletin de la Soci{\'e}t{\'e} Math{\'e}matique de France},
 journal = {Bull. Soc. Math. Fr.},
 issn = {0037-9484},
 volume = {62},
 pages = {193--224},
 year = {1934},
 language = {French},
 doi = {10.24033/bsmf.1221},
 url = {https://eudml.org/doc/86632},
 zbMATH = {3015894},
 Zbl = {0010.37701}
}

@article{monVicZap12,
 author = {Monzner, Alexandra and Vichery, Nicolas and Zapolsky, Frol},
 title = {Partial quasimorphisms and quasistates on cotangent bundles, and symplectic homogenization},
 fjournal = {Journal of Modern Dynamics},
 journal = {J. Mod. Dyn.},
 issn = {1930-5311},
 volume = {6},
 number = {2},
 pages = {205--249},
 year = {2012},
 language = {English},
 doi = {10.3934/jmd.2012.6.205},
 zbMATH = {6094024},
 Zbl = {1258.53091}
}

@article{Masl,
 author = {Maslov, Viktor Pavlovitch},
 title = {Théorie des perturbations et méthodes asymptotiques / V.P. Maslov ; suivi de deux notes complémentaires de V.I. Arnol'd et V.C. Bouslaev ; traduit par J. Lascoux,... R. Seneor,... ; préface de J. Leray,...},
 fjournal = {},
 journal = {},
 issn = {},
 volume = {},
 number = {},
 pages = {384pp},
 year = {1972},
 language = {French},
 doi = {},
 zbMATH = {},
 Zbl = {}
}

@article {Hor,
    AUTHOR = {H{\"o}rmander, Lars},
     TITLE = {Fourier integral operators. {I}},
   JOURNAL = {Acta Math.},
  FJOURNAL = {Acta Mathematica},
    VOLUME = {127},
      YEAR = {1971},
    NUMBER = {1-2},
     PAGES = {79--183},
      ISSN = {0001-5962,1871-2509},
   MRCLASS = {58G15 (35S05 47G05)},
  MRNUMBER = {388463},
MRREVIEWER = {Yu.\ V.\ Egorov},
       DOI = {10.1007/BF02392052},
       URL = {https://doi.org/10.1007/BF02392052},
}

@book {Car,
    AUTHOR = {Cardin, Franco},
     TITLE = {Elementary symplectic topology and mechanics},
    SERIES = {Lecture Notes of the Unione Matematica Italiana},
    VOLUME = {16},
 PUBLISHER = {Springer, Cham},
      YEAR = {2015},
     PAGES = {xviii+222},
      ISBN = {978-3-319-11025-7; 978-3-319-11026-4},
   MRCLASS = {70G45 (35F21 35Q41 37J05 58E05 70H03 70H05)},
  MRNUMBER = {3289033},
MRREVIEWER = {Maria\ Letizia\ Bertotti},
       DOI = {10.1007/978-3-319-11026-4},
       URL = {https://doi.org/10.1007/978-3-319-11026-4},
}

@article {Theret,
    AUTHOR = {Th\'eret, David},
     TITLE = {A complete proof of {V}iterbo's uniqueness theorem on
              generating functions},
   JOURNAL = {Topology Appl.},
  FJOURNAL = {Topology and its Applications},
    VOLUME = {96},
      YEAR = {1999},
    NUMBER = {3},
     PAGES = {249--266},
      ISSN = {0166-8641,1879-3207},
   MRCLASS = {53D12 (57R10 57R40 57R52)},
  MRNUMBER = {1709692},
MRREVIEWER = {Meike\ Akveld},
       DOI = {10.1016/S0166-8641(98)00049-2},
       URL = {https://doi.org/10.1016/S0166-8641(98)00049-2},
}

@article {ACZ,
    AUTHOR = {Conley, C. C. and Zehnder, E.},
     TITLE = {The {B}irkhoff-{L}ewis fixed point theorem and a conjecture of
              {V}. {I}. {A}rnol'd},
   JOURNAL = {Invent. Math.},
  FJOURNAL = {Inventiones Mathematicae},
    VOLUME = {73},
      YEAR = {1983},
    NUMBER = {1},
     PAGES = {33--49},
      ISSN = {0020-9910,1432-1297},
   MRCLASS = {58F05 (58C30)},
  MRNUMBER = {707347},
MRREVIEWER = {Yu.\ E.\ Gliklikh},
       DOI = {10.1007/BF01393824},
       URL = {https://doi.org/10.1007/BF01393824},
}

@incollection {Cha1,
    AUTHOR = {Chaperon, Marc},
     TITLE = {Quelques questions de g\'eom\'etrie symplectique (d'apr\`es,
              entre autres, {P}oincar\'e, {A}rnol'd, {C}onley et
              {Z}ehnder)},
 BOOKTITLE = {Bourbaki seminar, {V}ol. 1982/83},
    SERIES = {Ast\'erisque},
    VOLUME = {105-106},
     PAGES = {231--249},
 PUBLISHER = {Soc. Math. France, Paris},
      YEAR = {1983},
   MRCLASS = {58F05 (58-02)},
  MRNUMBER = {728991},
MRREVIEWER = {Richard\ C.\ Swanson},
}

@article {Cha2,
    AUTHOR = {Chaperon, Marc},
     TITLE = {Une id\'ee du type ``g\'eod\'esiques bris\'ees'' pour les
              syst\`emes hamiltoniens},
   JOURNAL = {C. R. Acad. Sci. Paris S\'er. I Math.},
  FJOURNAL = {Comptes Rendus des S\'eances de l'Acad\'emie des Sciences.
              S\'erie I. Math\'ematique},
    VOLUME = {298},
      YEAR = {1984},
    NUMBER = {13},
     PAGES = {293--296},
      ISSN = {0249-6291},
   MRCLASS = {58F05 (53C15)},
  MRNUMBER = {765426},
MRREVIEWER = {Karsten\ Grove},
}

@book {Wei,
    AUTHOR = {Weinstein, Alan},
     TITLE = {Lectures on symplectic manifolds},
    SERIES = {CBMS Regional Conference Series in Mathematics},
    VOLUME = {29},
      NOTE = {Corrected reprint},
 PUBLISHER = {American Mathematical Society, Providence, RI},
      YEAR = {1979},
     PAGES = {ii+48},
      ISBN = {0-8218-1679-9},
   MRCLASS = {58F05 (47G05 53C15 70Hxx)},
  MRNUMBER = {598470},
}

@article {Sik,
    AUTHOR = {Sikorav, Jean-Claude},
     TITLE = {Sur les immersions lagrangiennes dans un fibr\'e{} cotangent
              admettant une phase g\'en\'eratrice globale},
   JOURNAL = {C. R. Acad. Sci. Paris S\'er. I Math.},
  FJOURNAL = {Comptes Rendus des S\'eances de l'Acad\'emie des Sciences.
              S\'erie I. Math\'ematique},
    VOLUME = {302},
      YEAR = {1986},
    NUMBER = {3},
     PAGES = {119--122},
      ISSN = {0249-6291},
   MRCLASS = {58F05},
  MRNUMBER = {830282},
MRREVIEWER = {Mich\`ele\ Audin},
}

@article {Tray,
    AUTHOR = {Traynor, Lisa},
     TITLE = {Symplectic homology via generating functions},
   JOURNAL = {Geom. Funct. Anal.},
  FJOURNAL = {Geometric and Functional Analysis},
    VOLUME = {4},
      YEAR = {1994},
    NUMBER = {6},
     PAGES = {718--748},
      ISSN = {1016-443X,1420-8970},
   MRCLASS = {58E05 (58F05)},
  MRNUMBER = {1302337},
MRREVIEWER = {Karl\ Friedrich\ Siburg},
       DOI = {10.1007/BF01896659},
       URL = {https://doi.org/10.1007/BF01896659},
}

@article{carVit08,
 author = {Cardin, Franco and Viterbo, Claude},
 title = {Commuting {Hamiltonians} and {Hamilton}-{Jacobi} multi-time equations},
 fjournal = {Duke Mathematical Journal},
 journal = {Duke Math. J.},
 issn = {0012-7094},
 volume = {144},
 number = {2},
 pages = {235--284},
 year = {2008},
 language = {English},
 doi = {10.1215/00127094-2008-036},
 zbMATH = {5317180},
 Zbl = {1153.37029}
}

@article{ush19,
 author = {Usher, Michael},
 title = {Local rigidity, symplectic homeomorphisms, and coisotropic submanifolds},
 fjournal = {Bulletin of the London Mathematical Society},
 journal = {Bull. Lond. Math. Soc.},
 issn = {0024-6093},
 volume = {54},
 number = {1},
 pages = {45--53},
 year = {2022},
 language = {English},
 doi = {10.1112/blms.12555},
 zbMATH = {7729747},
 Zbl = {1523.53083}
}

@article{AGHIV23,
 author = {Asano, Tomohiro and Guillermou, St{\'e}phane and Humili{\`e}re, Vincent and Ike, Yuichi and Viterbo, Claude},
 title = {The {{\(\gamma\)}}-support as a micro-support},
 fjournal = {Comptes Rendus. Math{\'e}matique. Acad{\'e}mie des Sciences, Paris},
 journal = {C. R., Math., Acad. Sci. Paris},
 issn = {1631-073X},
 volume = {361},
 pages = {1333--1340},
 year = {2023},
 language = {English},
 doi = {10.5802/crmath.499},
 zbMATH = {7811799},
 Zbl = {1535.53082}
}

@article{bir32,
 author = {Birkhoff, George D.},
 title = {Sur quelques courbes ferm{\'e}es remarquables},
 fjournal = {Bulletin de la Soci{\'e}t{\'e} Math{\'e}matique de France},
 journal = {Bull. Soc. Math. Fr.},
 issn = {0037-9484},
 volume = {60},
 pages = {1--26},
 year = {1932},
 language = {French},
 doi = {10.24033/bsmf.1182},
 url = {https://eudml.org/doc/86601},
 zbMATH = {3006567},
 Zbl = {0005.22002}
}

@article{lec88,
 author = {Le Calvez, P.},
 title = {Propri{\'e}t{\'e}s des attracteurs de {Birkhoff}. ({Properties} of {Birkhoff} attractors)},
 fjournal = {Ergodic Theory and Dynamical Systems},
 journal = {Ergodic Theory Dyn. Syst.},
 issn = {0143-3857},
 volume = {8},
 number = {2},
 pages = {241--310},
 year = {1988},
 language = {French},
 doi = {10.1017/S0143385700004442},
 zbMATH = {4074001},
 Zbl = {0657.58009}
}

@article{AGIV24,
 author = {Asano, Tomohiro and Guillermou, St{\'e}phane and Ike, Yuichi and Viterbo, Claude},
 title = {Regular {Lagrangians} are smooth {Lagrangians}},
 year = {2025},
journal={J. Math. Soc. Japan},
pages={1-18},
doi={10.2969/jmsj/93799379 },
 howpublished = {Preprint, {arXiv}:2407.00395 [math.{SG}] (2024)},
 url = {https://arxiv.org/abs/2407.00395},
 arXiv = {arXiv:2407.00395}
}

@article{SymplHo,
 author = {Viterbo, C.},
 title = {Symplectic Homogenization},
 fjournal = {Journal de l’École polytechnique — Mathématiques},
 journal = {J. Éco. Pol. — Math.},
 issn = {},
 volume = {10},
 number = {},
 pages = {67--140},
 year = {2013},
 language = {English},
 doi = {10.5802/jep.214},
 zbMATH = {},
 Zbl = {}
}

@article {FS07,
    AUTHOR = {Frauenfelder, Urs and Schlenk, Felix},
     TITLE = {Hamiltonian dynamics on convex symplectic manifolds},
   JOURNAL = {Israel J. Math.},
  FJOURNAL = {Israel Journal of Mathematics},
    VOLUME = {159},
      YEAR = {2007},
     PAGES = {1--56},
      ISSN = {0021-2172,1565-8511},
   MRCLASS = {53D40 (37J05 37J45)},
  MRNUMBER = {2342472},
MRREVIEWER = {Michael\ J.\ Usher},
       DOI = {10.1007/s11856-007-0037-3},
       URL = {https://doi.org/10.1007/s11856-007-0037-3},
}

@article {Lan16,
    AUTHOR = {Lanzat, Sergei},
     TITLE = {Hamiltonian {F}loer homology for compact convex symplectic
              manifolds},
   JOURNAL = {Beitr. Algebra Geom.},
  FJOURNAL = {Beitr\"age zur Algebra und Geometrie. Contributions to Algebra
              and Geometry},
    VOLUME = {57},
      YEAR = {2016},
    NUMBER = {2},
     PAGES = {361--390},
      ISSN = {0138-4821,2191-0383},
   MRCLASS = {53D05 (53D40 53D45)},
  MRNUMBER = {3493995},
MRREVIEWER = {Umberto\ Leone\ Hryniewicz},
       DOI = {10.1007/s13366-015-0254-6},
       URL = {https://doi.org/10.1007/s13366-015-0254-6},
}
\end{document}